\documentclass[11pt]{amsart}
\usepackage{fullpage}
\usepackage{amssymb,amscd,mathrsfs}
\usepackage{hyperref}
\usepackage{amsmath} 
\usepackage{tikz-cd,extarrows,enumitem,xfrac}

\usepackage[all,cmtip]{xy}
\usepackage{color}

\numberwithin{equation}{section}
\newtheorem{thm}[equation]{Theorem}
\newtheorem{lem}[equation]{Lemma}
\newtheorem{prop}[equation]{Proposition}
\newtheorem{cor}[equation]{Corollary}

\theoremstyle{definition}
\newtheorem{defn}[equation]{Definition}

\theoremstyle{remark}
\usepackage{enumitem}  
\newtheorem{remark}[equation]{Remark}

\newtheorem{ex}[equation]{Example}
\newtheorem{rem}[equation]{Remark}

\newcommand{\one}{\ensuremath{\mathbf{v}}}
\newcommand{\two}{\ensuremath{\mathbf{w}}}

\renewcommand{\L}{\mathcal{L}}
\newcommand{\F}{\mathbb{F}}

\DeclareMathOperator{\Ima}{Im}
\def\Tor{\mathrm{Tor}}

\def\Sp{\mathrm{Sp}}
\def\CAlg{\mathrm{CAlg}}

\def\cS{\mathcal{S}}
\def\niceH{\mathcal{H}}

\def\Z{\mathbb{Z}}
\def\Q{\mathbb{Q}}

\def\leq{\leqslant}
\def\geq{\geqslant}

\def\ra{\rightarrow}

\def\vol{{\sf vol}}

\def\HH{\mathsf{HH}}
\def\THH{\mathsf{THH}}

\def\ie{\emph{i.e.}}

\def\id{\mathrm{id}}

\begin{document}
\title{Loday constructions on twisted products and on tori} 
\author[Hedenlund]{Alice Hedenlund}
\address{Department of Mathematics, University of Oslo, Box 1053,
  Blindern, NO - 0316 Oslo, Norway} 
\email{aliceph@math.uio.no}
\author[Klanderman]{Sarah Klanderman}
\address{Department of Mathematics, Michigan State University, 619 Red Cedar Rd,
    East Lansing, MI 48840, USA}
\email{klander2@msu.edu}
\author[Lindenstrauss]{Ayelet Lindenstrauss}
\address{Department of Mathematics, Indiana University, 831 East 3rd Street,
  Bloomington, IN 47405, USA}
\email{alindens@indiana.edu}
\author[Richter]{Birgit Richter}
\address{Fachbereich Mathematik der Universit\"at Hamburg,
Bundesstra{\ss}e 55, 20146 Hamburg, Germany}
\email{birgit.richter@uni-hamburg.de}
\author[Zou]{Foling Zou}
\address{Department of Mathmatics, University of Chicago, 5734 S
  University Ave, Chicago, IL 60637, USA} 
\email{foling@math.uchicago.edu}
\date{\today}

\keywords{torus homology, (higher) Hochschild homology, (higher)
  topological Hochschild homology, stability,
  twisted Cartesian products}
\subjclass[2000]{Primary 18G60; Secondary 55P43}
\begin{abstract}
  We develop a spectral sequence for the homotopy groups of Loday
  constructions with respect to twisted products in the case where the
  group involved is a constant simplicial group. We show that for
  commutative Hopf algebra spectra Loday constructions are stable,
  generalizing a result by Berest, Ramadoss and Yeung. 
  We prove that several truncated polynomial rings are
  \emph{not} multiplicatively stable by investigating their torus homology. 
\end{abstract}
\maketitle

\section*{Introduction}
When one studies commutative rings or ring spectra, important
homology theories are topological Hochschild or its higher
versions. These are specific examples of the Loday construction, whose
definition relies  
on the fact that commutative ring spectra are enriched 
in simplicial sets: for a simplicial set $X$ and a commutative ring
spectrum $R$ one can define 
the tensor $X \otimes R$ as a simplicial spectrum whose $n$-simplices are
\[ \bigwedge_{x \in X_n} R. \]  By slight abuse of notation $X \otimes
R$ also denotes the commutative ring spectrum that is 
the geometric realization of this simplicial spectrum. This
recovers topological 
Hochschild homology of $R$, $\THH(R)$, when $X=S^1$, and 
higher topological Hochschild homology, $\THH^{[n]}(R)$, for higher dimensional
spheres $S^n$. Tensoring
satisfies several properties \cite[VII, \S 2, \S 3]{ekmm}, two of which are: 
\begin{itemize}
\item If $X$ is
a homotopy pushout, 
$X = X_1 \cup^h_{X_0} X_2$, then the tensor product of $R$ with $X$
splits as a homotopy 
pushout in the category of commutative ring spectra which is the
derived smash product:  
\[ (X_1 \cup^h_{X_0} X_2) \otimes R \simeq (X_1 \otimes R)
  \wedge^L_{(X_0 \otimes R)} (X_2 \otimes R).\]

\item A product of simplicial sets $X \times Y$ gives rise to an iterated
tensor product: 
\[ (X \times Y) \otimes R \simeq X \otimes (Y \otimes R). \]
\end{itemize}

This last expression does not, however,  imply that calculating the
homotopy groups of $(X \times Y) \otimes R$ 
is easy. In particular, if one iterates the trace map from algebraic
K-theory to topological 
Hochschild homology $n$ times, one obtains a map
\[ K^{(n)}(R) = \underbrace{K(K(\ldots(K}_{n}(R))\ldots )) \ra
  \underbrace{(S^1 \times \ldots \times S^1)}_n \otimes R .\]
Since iterated K-theory is of interest 
in the context of chromatic red-shift, one would like to know as much
about $(S^1 \times \ldots \times S^1) \otimes R$ as possible.

In some good cases, the homotopy type of $X \otimes R$ only depends on the
suspension of $X$ in the sense that if $\Sigma X \simeq \Sigma Y$,
then one has $X \otimes R \simeq Y \otimes R$. 
This property is called \emph{stability}. Stability for instance holds for
Thom spectra $R$ that arise from an infinite loop map to the
classifying space $BGL_1(S)$ (see Theorem 1.1 of \cite{sch}), or for
$R = KU$ or $R = KO$ \cite[\S 4]{lr}.  

One can also work relative to a fixed commutative ring spectrum $R$
and consider commutative 
$R$-algebra spectra $A$ and ask whether $X \otimes_R A$ only depends
on the homotopy type of 
$\Sigma X$. In this paper, we will often work with coefficients: we
look at pointed simplicial 
sets $X$ and place a commutative $A$-algebra spectrum $C$ at the
basepoint of $X$. In other words, when $X$ is pointed then the
inclusion of the basepoint makes $X \otimes_R A$  into a commutative
$A$-algebra and we can look at  $\L_X^R(A;C) =(X \otimes_R A)\wedge_A
C$, the Loday construction with respect to $X$ of $A$ over $R$ with
coefficients in $C$. We call the pair $(A; C)$ \emph{stable} if the
homotopy type of 
$\L_X^R(A; C)$ only depends on the homotopy type of $\Sigma X$. Note
that the ring $R$ is not part the 
notation when we say that $(A;C)$ is stable although the question depends on the choice of $R$, so the context should specify the $R$ we are working over. We call the commutative $R$-algebra $A$
\emph{multiplicatively stable} as in 
\cite[Definition 2.3]{lr} if $\Sigma X \simeq \Sigma Y$ implies that
$\L_X^R(A) \simeq \L_Y^R(A)$ 
as commutative $A$-algebra spectra. If $A$ is multiplicatively stable,
then for any cofibrant 
commutative $A$-algebra $C$, the pair $(A; C)$ is stable (see
\cite[Remark 2.5]{lr}).

We investigate several algebraic examples, \ie, commutative ring
spectra that are 
Eilenberg Mac Lane spectra of commutative rings. For instance we show
that the pairs $(H\Q[t]/t^m; H\Q)$ are not stable for all $m \geq 2$,
extending a result by Dundas and Tenti 
\cite{DT}. We also prove integral and mod-$p$ versions of this result.

Work of Berest, Ramadoss and Yeung implies that the homotopy types of
$\L_X^{Hk}(HA; Hk)$ 
and $\L_X^{Hk}(HA)$ only depend on the homotopy type of $\Sigma X$ if
$k$ is a field and if 
$A$ is a commutative Hopf algebra over $k$. We generalize this result
to commutative Hopf algebra spectra. 

Moore introduced twisted cartesian products as simplicial models for
fiber bundles. We develop a 
Serre type spectral sequence for Loday constructions of twisted
cartesian products where the 
twisting is governed by a constant simplicial group. As a concrete
example we compute the Loday 
construction with respect to the Klein bottle for a polynomial algebra
over a field with 
characteristic not equal to $2$.

\subsection*{Content}
In Section \ref{sec:loday} we recall the definition of the Loday
construction and fix notation. 
Section \ref{sec:tcp} contains the construction of a spectral sequence
for the homotopy groups of 
Loday constructions with respect to twisted cartesian
products. Our results on commutative Hopf algebra spectra can be found
in Section \ref{sec:hopf}. In Section \ref{sec:trunc} we 
prove that truncated 
polynomial algebras of the form $\Q[t]/t^m$ and $\Z[t]/t^m$ for $m
\geq 2$ are not multiplicatively 
stable by
comparing the Loday construction of tori to the Loday construction of
a bouquet of spheres 
corresponding to the cells of the tori. We also show that for $2 \leq
m < p$ the $\F_p$-algebra 
$\F_p[t]/t^m$ is not stable. 
\subsection*{Acknowledgements}
We thank the organizers of the third Women in Topology workshop, Julie
Bergner, Ang\'elica Osorno, and  Sarah Whitehouse, and  also the
Hausdorff Institute of Mathematics for their hospitality during the
week of the workshop. We thank the Hausdorff Research Institute for
Mathematics, grants NSF-DMS 1901795 and NSF-HRD 1500481---AWM ADVANCE,
and the Foundation Compositio Mathematica for their support of the
workshop.   We thank  
Maximilien P\'eroux for help with coalgebras in spectra, Inbar Klang for
a helpful remark about norms, Mike Mandell for a helpful discussion on
$E_n$-spaces, Jelena Grbi\'c for pointing out  
\cite{Schaper} to us, and Thomas Nikolaus for $\infty$-category  
support. AL was supported by Simons Collaboration Grant 359565. The last
two authors thank the Department of Mathematics at Indiana University
for its hospitality and BR thanks the Department of 
Mathematics at Indiana University for support as a short-term research
visitor in 2019.

\section{The Loday construction: basic features} 
\label{sec:loday}

We recall some definitions
concerning the Loday construction and we fix notation. 

For our work we can use any good symmetric monoidal category of
spectra whose category of commutative monoids is Quillen equivalent to
the category of $E_\infty$-ring spectra, such as symmetric
spectra \cite{HSS}, orthogonal spectra \cite{mm} or $S$-modules
\cite{ekmm}. As parts of the paper require us to work with a specific
model category we chose to work with the category of $S$-modules. 

Let $X$ be a finite pointed simplicial set and let $R \ra A \ra C$ be a sequence
of maps of commutative ring spectra. 

\begin{defn} \label{def:loday}
The \emph{Loday construction with respect to $X$ of
$A$ over $R$ with coefficients in $C$} is the simplicial commutative
augmented $C$-algebra spectrum $\L^R_X(A; C)$ given by 
\[ \L^R_X(A; C)_n = C \wedge \bigwedge_{x \in X_n \setminus *} A \]
where the smash products are taken over $R$.  
Here, $*$ denotes the basepoint of $X$ and we place a copy of $C$ at
the basepoint. 
The simplicial structure of $\L^R_X(A; C)$ 
is straightforward: Face maps $ d_i$ on $X$ induce multiplication in
$A$ or the $A$-action on $C$ if the basepoint is 
involved. Degeneracies $s_i$ on $X$ correspond to the insertion of
the unit maps $\eta_A \colon R \ra A$ over all $n$-simplices which
are not hit by $s_i\colon X_{n-1} \to X_n$.
\end{defn}

As defined above, $\L^R_X(A; C)$ is a simplicial commutative augmented
$C$-algebra spectrum. In the following we will always assume that $R$
is a cofibrant commutative $S$-algebra, $A$ is a cofibrant commutative
$R$-algebra and $C$ is a cofibrant commutative $A$-algebra. This
ensures that the homotopy type of $\L^R_X(A; C)$ is well-defined and
depends only on the homotopy type of $X$. 

\begin{remark} \label{originalLoday}
When $R \ra A \ra C$ is a sequence 
of maps of commutative rings, we can of course use the above
definition for $HR\to HA\to HC$. The original construction by
Loday \cite[Proposition 6.4.4]{loday} used  
\[C \otimes \bigotimes _{x \in X_n \setminus *} A \]
instead 
with the tensors taken over $R$ as the $n$-simplices in $\L^R_X(A;
C)$. 

This algebraic definition also makes sense if $R$ is a commutative
ring and $A\to C$ is a map of commutative simplicial $R$-algebras. 

It continues to work if $R$ is a commutative ring and $A\to C$ is a
map of graded-commutative $R$-algebras, with the $n$-simplices defined
as above, but the maps between them require a sign correction as terms
are pulled past each other---see \cite[Equation (1.7.2)]{pira-hodge}. 
\end{remark}

An important case is $X = S^n$. In this case we write $\THH^{[n],R}(A;
C)$ for $\L^R_{S^n}(A; C)$; this is the \emph{higher order topological
Hochschild homology of order $n$ of $A$ over $R$ with coefficients in
$C$}.

Let $k$ be a commutative ring, $A$ be a commutative $k$-algebra,
and $M$ be an $A$-module.  Then we define 
\[
  \THH^{[n],k}(A;M) := \L^{Hk}_{S^n}(HA;HM).
\]
If $A$ is flat over $k$, then $\pi_*\THH^{k}(A;M) \cong
\HH_*^{k}(A;M)$ \cite[Theorem IX.1.7]{ekmm} and this also holds for
higher order Hochschild homology in the sense of Pirashvili
\cite{pira-hodge}: $\pi_*\THH^{[n],k}(A;M) \cong  \HH_*^{[n],k}(A;M)$
if $A$ is $k$-flat \cite[Proposition 7.2]{blprz}. 

Given a commutative ring $A$ and an element
$a\in A$, we write $A/a$ instead of $A/(a)$.

\section{A spectral sequence for twisted cartesian products}
\label{sec:tcp} 
We will start by letting $R\to A$ be a map of commutative rings and we
study Loday constructions $ \L^R_B(A ^\tau)$ over a finite
simplicial set $B$, where $\tau$ indicates a twisting by a discrete
group $G$ that acts on $A$ via ring isomorphisms.  This construction
can be adapted as in Definition \ref{def:loday} and Remark
\ref{originalLoday}  to allow coefficients in an $A$-algebra $C$ if $B$
is 
pointed, and to the case  where $R\to A$ is a map of commutative ring
spectra, or $R$ is a commutative ring and $A$ is a graded-commutative
$R$-algebra 
or a simplicial commutative  $R$-algebra. 

If we have a twisted cartesian product (TCP) in the sense of 
\cite[Chapter IV]{may}  $E(\tau)=F\times_\tau B$ where the fiber $F$
is a simplicial $R$-algebra and the simplicial structure group $G$
acts on $F$ by simplicial $R$-algebra isomorphisms, it is possible to
generalize this definition of the Loday construction to  allow
twisting by a simplicial structure group, as expained in Definition \ref{def:twistedloday} below.

We show an example where such a TCP arises: if we start with a TCP of
simplicial sets $E(\tau) =F\times_\tau B$ with twisting in a
simplicial structure group $G$ acting on $F$ simplicially on the left
and with a map of commutative rings $R\to A$, we can use that twisting
to construct a TCP with fiber equal to the simplicial commutative $R$-algebra
$\L^R_F(A)$ and with the structure group $G$ acting on $\L^R_F(A)$ by
$R$-algebra isomorphisms.  In that situation, we get that 
\[\L^R_{E(\tau)} (A)\cong \L^R_B( \L^R_F(A) ^\tau),\]
which generalizes the fact that for a product, $\L^R_{F\times B}
(A)\cong \L^R_B( \L^R_F(A) ).$ 
If the structure group $G$ is discrete, \ie, if $G$ is a
constant simplicial group, $\L^R_{E(\tau)} (A)$ can be written as a
bisimplicial set and we get a spectral sequence for calculating its
homotopy groups.  

\begin{defn} \label{def:twistedloday} Let $B$  be a finite simplicial
  set, $R$ be a commutative ring, and $A$ be a commutative $R$-algebra
  (or a graded-commutative $R$-algebra, or a simplicial commutative 
  $R$-algebra).  Let $G$ be a discrete group acting on $A$ from the
  left via isomorphisms of $R$-algebras, and let $\tau$ be a function
  from the positive-dimensional simplices of $B$ to $G$ so that 
\begin{equation}\label{tauconst} 
  \begin{array}{rcll}
 \tau (b) & = & [\tau ( d_0 b)]^{-1} \tau( d_1 b)
  & \text{ for }  q>1, b \in B_q, \\ 
   \tau ( d_{i}b)  & = & \tau(b) & \text{ for } i\geq 2, q>1, b \in B_q,\\
\tau ( s_{i} b) & = & \tau ( b) & \text{ for } i \geq 1, q>0, b \in
                                  B_q, \text{ and } \\
    \tau (s_0 b) &= &e_G & \text{ for } q>0, b \in B_q.
\end{array}
\end{equation}
The \emph{twisted Loday construction with respect to $B$ of
$A$ over $R$ twisted by $\tau$} is the simplicial commutative (resp.,
graded-commutative, or bisimplicial commutative) $R$-algebra  
 $\L^R_B(A^\tau)$ given by
\[\L^R_B(A^\tau)_n = \L^R_{B_n}(A) =\bigotimes _{b \in B_n} A\]
where the tensor products are taken over $R$, with
\begin{align}
     d_0\left(\bigotimes_{b\in B_n} f_b\right)  =\bigotimes_{c\in B_{n-1}}
  g_c  &\text{ with } g_c=\prod_{b:  d_0 b = c}
         \tau(b)( f_b), \nonumber \\
      d_i\left(\bigotimes_{b\in B_n} f_b\right)  =\bigotimes_{c\in B_{n-1}}
  g_c  & \text{ with }  g_c=\prod_{b:  d_i b = c}  
         f_b \text{ for } \ 1 \le i \le n,   \text{ and
         } \nonumber\\ 
     s_i\left(\bigotimes_{b\in B_n} f_b\right)  = \bigotimes_{d\in B_{n+1}} h_d &
\text{ with }  h_d=\prod_{b: s_i b = d} f_b \text{ for }  0 \le i
         \le n.\nonumber  
\end{align}
 \end{defn}
We should think of the copy of $A$ sitting over a simplex $b\in B_n$
as sitting over its $0$th vertex, and of $\tau(b)$ as translating
between the $A$ over $b$'s $0$th vertex and the $A$ over $b$'s $1$st
vertex. 

\begin{lem} The definition above makes $\L^R_B(A^\tau)$ into a simplicial set.
\end{lem}
\begin{proof}
To check this we need only check the
relations involving $ d_0$, since the ones that do not involve
$\tau$ work in the same way that they do in the usual Loday
construction.  For $j>1$, we get $d_0 d_j= d_{j-1} d_0$ because in
both terms, for
any $c\in B_{n-2}$ we get the product over all $b\in B_n$ with
$ d_0 d_j b= d_{j-1} d_0 b = c$ of terms that
are either $\tau(b)(f_b)$ or $\tau( d_jb)(f_b)$. These are the
same by the condition in Equation (\ref{tauconst}) above.  For $j=1$,
we 
get the product over all $b\in B_n$ with $ d_0 d_1
b= d_{0} d_0 b = c$ of terms that are either
$\tau( d_1b)(f_b)$ or $\tau( d_0b)\tau(b)(f_b)$, which
again agree by Equation (\ref{tauconst}).  We get
$ d_0s_0=\mathrm{id}$ since $\tau(s_0b)=e_G$, and
$ d_0s_i=s_{i-1} d_o$ for $i>0$ since for those $i$,
$\tau(s_ib)=\tau(b)$.   
\end{proof}

Following Moore, May considers the following
simplicial version of a fiber bundle \cite[Definition 18.3]{may}: 
\begin{defn} \label{def:TCP}
Let $F$ and $B$ be simplicial sets and let $G$ be a simplicial group
which acts on $F$  from the left.  Let
 $\tau \colon B_q\to G_{q-1}$ for all $q>0$ be functions so that 
\[ \begin{array}{rcll}
     d_0 \tau (b) & = & [\tau ( d_0 b)]^{-1}
                              \tau( d_1 b)
     & \text{ for }  q>1, b \in B_q, \\
\tau (d_{i+1}b)    & = &  d_i \tau(b)  & \text{ for }  i \ge 1,
         q>1,  b \in B_q, \\
\tau (s_{i+1} b)   & = & s_i \tau (b) & \text{ for } i \ge 0, q > 0, b
                                        \in B_q, \text{ and }\\
    \tau (s_0 b) & = & e_q &\text{ for }  q>0, b \in B_q.
   \end{array} \]

The \emph{twisted Cartesian
  product (TCP)} $E(\tau)=F\times_\tau B$ is the simplicial set  whose
$n$-simplices are given by 
\[E(\tau)_n= F_n \times B_n, \]
 with simplicial structure maps
\begin{enumerate}[label=(\roman*)]
\item $ d_0(f,b) = (\tau (b) \cdot  d_0 f,  d_0 b)$,
\item $ d_i(f, b) = ( d_i f,  d_i b)$ \  $\forall
  i>0$, \text{ and }
\item $s_i (f, b) = (s_i f, s_i b)$ \ $\forall i \ge 0$.
\end{enumerate}
These structure maps satisfy the necessary relations to be a simplicial set because of
the conditions that $\tau$ satisfies. 
\end{defn}

\begin{defn} \label{def:twistedlodayTCP} If $R$ is a commutative ring
  and $E(\tau)=C\times_\tau B$ is a TCP as in Definition \ref{def:TCP}
  where $C$ is a commutative simplicial $R$-algebra and the simplicial
  group $G$ acts on $C$ by $R$-algebra isomorphisms (that is, for
  every $q\ge 0$, the group $G_q$ acts on the commutative $R$-algebra
  $C_q$ by $R$-algebra isomorphisms) then we can use the twisting
  $\tau$ to define the \emph{twisted Loday construction with respect to $B$
  of $C$ over $R$, twisted by $\tau$}, 
\[ \L^R_B(C^\tau)_n =\L_{B_n}^R (C_n)= \bigotimes _{b \in B_n} C_n \]
with twisted structure maps given on monomials $\bigotimes_{b\in B_n}
f_b$, with $f_b\in C_n$ for all $b\in B_n$, by 
\begin{align}
     d_0\left(\bigotimes_{b\in B_n} f_b\right)  =\bigotimes_{c\in B_{n-1}}
  g_c  &\text{ with } g_c=\prod_{b:  d_0 b = c}
         \tau(b)( d_0 f_b), \nonumber \\
      d_i\left(\bigotimes_{b\in B_n} f_b\right)  =\bigotimes_{c\in B_{n-1}}
  g_c  & \text{ with }  g_c=\prod_{b:  d_i b = c}  d_i
         f_b \text{ for } \ 1 \le i \le n,   \text{ and
         } \label{deftwist:simpalg}\\ 
     s_i\left(\bigotimes_{b\in B_n} f_b\right)  = \bigotimes_{d\in B_{n+1}} h_d &
\text{ with }  h_d=\prod_{b: s_i b = d}s_i f_b \text{ for }  0 \le i
         \le n.\nonumber  
\end{align}  
\end{defn}
Note that there are two sets of simplicial structure maps being used,
those of $C$ inside and those of $B$ outside.  This looks like the
diagonal of a bisimplicial set, but since our twisting $\tau \colon B_{q+1}
\to G_q$ explains only how to twist elements in $C_q$, this is not the
case  unless the structure group $G$ is a discrete group, viewed as a
constant simplicial group.  
\smallskip

If the structure group $G$ is discrete, there is overlap between
Definition \ref{def:twistedloday} and 
Definition \ref{def:twistedlodayTCP}.  The simplicial commutative
$R$-algebra case  
of Definition \ref{def:twistedloday} actually gives a bisimplicial
set: we use only the simplicial structure of $B$ in the definition and
if $A$ also has simplicial structure, that remains untouched.   The
diagonal of that bisimplicial set agrees with the constant simplicial
group case of Definition \ref{def:twistedlodayTCP}. 

\medskip
Given any TCP of simplicial sets $E(\tau)=F\times_\tau B$ as in
Definition \ref{def:TCP} and a map $R\to A$ of commutative rings, we
can construct $\L^R_F(A) \times_\tau B$ which is a TCP of commutative
simplicial algebras $R$-algebras as 
in Definition \ref{def:twistedlodayTCP} 
using the same structure group $G$ and twisting function $\tau\colon B_q\to
G_{q-1}$.  We use the simplicial left action of $G_n$ on $F_n$ which
we denote by $(g, f)\mapsto gf$ to obtain a left action by simplicial
$R$-algebra isomorphisms 
\begin{align}
G_n \times \L^R_{F_n}(A) &  \to  \L^R_{F_n}(A)\nonumber \\
(g, \bigotimes_{f\in F_n} a_f) & \mapsto \bigotimes_{f\in F_n} a_{g^{-1}f}.
\end{align}
Since the original action of $G_n$ on $F_n$ was a left action, this is
a left action. In the original monomial, the $f$th coordinate is
$a_f$.  After $g\in G_n$ acts on it, the $f$th coordinate is
$b_f=a_{g^{-1} f}$.  After $h\in G_n$ acts on the result of the action
of $g$, the $f$th coordinate is $b_{h^{-1} f} =a_{g^{-1} h^{-1} f}$,
which is the same as the result of acting by $hg$ on the 
monomial. 

\begin{prop}\label{LodayconstructionTCP}
If $E(\tau)=F\times_\tau B$ is a TCP and $R\to A$ is a map of
commutative rings, and we use the simplicial set twisting function
$\tau$ to construct a simplicial $R$-algebra twisting function to
obtain a TCP $\L_F^R(A) \times_\tau B$ as above, we get that 
\[\L_{E(\tau)}^R(A)   \cong \L^R_B(\L_F^R(A)^\tau). \]
\end{prop}
This uses the definition of the Loday construction of a simplicial
algebra twisted by a simplicial group in Definition
\ref{def:twistedlodayTCP}.   

Proposition \ref{LodayconstructionTCP} generalizes the well-known fact
that for a product of simplicial sets,
\[\L_{F\times B}^R(A)\cong \L_B^R(\L_F^R(A)). \]
\begin{proof}
Both $\L_{E(\tau)}^R(A)$ and $ \L^R_B(\L_F^R(A)^\tau)$ have the same
set of $n$-simplices for every $n\geq 0$:
\[
\bigotimes_{e\in E(\tau)_n } A = \bigotimes_{(f,b)\in F_n\times B_n  }
A \cong \bigotimes_{b\in B_n}\ \  (\bigotimes_{f\in F_n} \ \ A). \]
We have to show that the simplicial structure maps agree  with
respect to this identification. 

For $1\leq  i\leq n$, for any choice of elements $x_{(f,b)} \in A$,
\[  d_i\left(\bigotimes_{(f,b)\in F_n\times B_n  } x_{(f,b)}\right) =
  \bigotimes_{(g,c)\in F_{n-1}\times B_{n-1}  } y_{(g,c)}\] 
where  
\[y_{(g,c)}=\prod_{(f,b): ( d_i f,  d_i b)=(g,c) }
  x_{(f,b)} =\prod _{b:  d_i b=c} \left(\prod_{f:  d_i
    f = g}  x_{(f,b)}\right). \] 
The internal product on the right-hand side is what we get from
$ d_i$ on $\L_F^R(A)$ and the external product is what we get
from $ d_i$ of $\L_B^R$, so this agrees with the definition in
Equation (\ref{deftwist:simpalg}). 

The proof that the $s_i$, $0\leq i\leq n$ agree is very similar.  

The interesting case is that of $ d_0$.  For any choice of
elements $x_{(f,b)} \in A$, the boundary $ d_0$ associated to
$\L_{E(\tau)}^R(A)$ satisfies 
\begin{equation}
\label{oneway}
 d_0 \left(\bigotimes_{(f,b)\in F_n\times B_n  } x_{(f,b)}\right) =
\bigotimes_{(g,c)\in F_{n-1}\times B_{n-1}  } y_{(g,c)}, 
\end{equation}
where
\[ y_{(g,c)}=\prod_{(f,b):  d_0(f, b)=(g,c) } x_{(f,b)} =\prod
  _{b:  d_0 b=c} \left( \prod_{f: \tau(b) \cdot  d_0 f= g}
  x_{(f,b)}\right). \] 
From the $ \L_B^R(\L_F^R(A))$ point of view, by Equation
\eqref{deftwist:simpalg}. 
\begin{align*}
   d_0\left(\bigotimes_{b\in B_n} (\bigotimes_{f\in F_n} x_{(f,b)})\right) &
 = \bigotimes_{c\in B_{n-1}} \prod_{b:   d_0 b = c}
  \tau(b)  d_0\left(\bigotimes_{f\in F_n} x_{(f,b)}\right) \\ 
  & = \bigotimes_{c\in B_{n-1}}  \prod_{b:   d_0 b = c} \tau(b)
  \left(\bigotimes_{g\in F_{n-1}} \prod_{f:  d_0 f = g} x_{(f,b)}\right)  \\
  & = \bigotimes_{c\in B_{n-1}}  \prod_{b:  d_0 b = c}
  \left(\bigotimes_{g\in F_{n-1}}  \prod_{f:   d_0 f = \tau(b)^{-1} g}
    x_{(f,b)}\right)\\  
& = \bigotimes_{(g,c)\in F_{n-1}\times B_{n-1}} \prod_{b:   d_0
 b=c} \left(\prod_{f:  d_0 f = \tau(b)^{-1} g} x_{(f,b)}\right),  
 \end{align*}
which is exactly what we got in  \eqref{oneway}.
\end{proof}
If $G$ is a discrete group and $E(\tau)$ is constructed using $G$,
then for every $q > 0$ there is a function
$\tau \colon B_q\to G$ satisfying  
the conditions listed in Equation \eqref{tauconst} and $G$ acts
simplicially on $F$ on the left.  
\begin{thm} \label{thm:tcpspsq}
If $E(\tau)=F\times_\tau B$ is a TCP where the twisting is by a
constant simplicial group $G$   and if $R\to A$ is a map of
commutative rings 
so that $\pi_*( \L_F^R(A))$ is flat over $R$, then there is a spectral sequence 
\begin{equation}\label{eq:tcpspsq}
  E^2_{p, q} =\pi_p((\L_{B}^R(\pi_* \L_F^R(A)^\tau))_q) \Rightarrow
  \pi_{p+q} (\L_{E(\tau)}^R(A)).
  \end{equation}
\end{thm}
Here, $\pi_* \L_F^R(A)$ is a graded commutative $R$-algebra.  For any
fixed $p$ and $q$, we consider the degree $q$ part of
$\L_{B_p}^R(\pi_* \L_F^R(A)^\tau)$, $(\L_{B_p}^R(\pi_*
\L_F^R(A)^\tau))_q$. This forms a simplicial abelian group which in
degree $p$ is  
$(\L_{B_p} ^R(\pi_* \L_F^R(A)))_q$, with simplicial structure maps
induced by those of $B$ with the twisting by $\tau$, and
$\pi_p((\L_{B}^R(\pi_* \L_F^R(A)^\tau))_q)$ denotes its $p$th
homotopy group.  
The flatness assumption above is for instance satisfied if $R$ is a
field. 
\begin{proof}
Since the twisting is by a constant simplicial group $G$, we are able
to form a bisimplicial $R$-algebra 
\begin{equation}
\label{def:bicomp}
(m,n)\mapsto \bigotimes_{b\in B_m}  \bigotimes_{f\in F_n} A.
\end{equation}
In the $n$-direction, the simplicial structure maps $ d_i^F$ and
$s_i^F$ will simply be the simplicial structure maps of the Loday
construction $\L^R_F(A)$ applied simultaneously to all the copies of
$\L^R_F(A)$ over all the $b\in B_n$.  In the $m$ direction,
$ d_i^B$ and $s_i^B$ are the simplicial structure maps of the
twisted Loday construction, as in Equation (\ref{tauconst}) in
Definition  \ref{def:twistedloday}.  These commute exactly because the
simplicial structure maps in $G$ are all equal to the identity. For any choice
of $x_b\in \L^R_F(A)_n$ for all $b\in B_m$,  
\[  d^B_0 d^F_i \left(\bigotimes_{b\in B_m} x_b\right) =
 d^B_0 \left(\bigotimes_{b\in B_m}  d^F_i ( x_b)\right)=
\bigotimes_{c\in B_{m-1}}  \prod_{b:  d_0b=c} \tau(b)  d_i^F(x_b)
\]
while
\[  d^F_i   d^B_0(\bigotimes_{b\in B_m} x_b) =
 d^F_i \left(\bigotimes_{c\in B_{m-1}}  \prod_{b:  d_0b=c}
\tau(b)\cdot x_b\right) 
= \bigotimes_{c\in B_{m-1}} \prod_{b:  d_0b=c}
 d^F_i( \tau(b)\cdot x_b), 
\]
which is the same since 
\[ d^F_i ( \tau(b)\cdot x_b) =  d_i ( \tau(b)) \cdot
 d^F_i  (x_b) = \tau(b) \cdot   d^F_i  (x_b). \]
Note that since the twisting is by a constant simplicial group,
$\L_{E(\tau)}^R(A)   \cong \L^R_B(\L_F^R(A)^\tau)$ is exactly the
diagonal of the bisimplicial $R$-algebra in Equation
(\ref{def:bicomp}).

\smallskip
We use the standard result (see for instance
\cite[Theorem 2.4 of Section IV.2.2]{gj})  that the total
complex of a bisimplicial abelian group with the alternating sums of
the vertical and the horizontal face maps is chain homotopy equivalent
to the usual chain complex associated to the diagonal of that
bisimplicial abelian group.  Since we know that the realization of the
diagonal is homeomorphic to the double realization of the bisimplicial
abelian group, in order to know the homotopy groups of the double
realization of a bisimplicial abelian group, we can calculate the
homology of its total complex with respect to the alternating sums of
the vertical and the horizontal face maps.  Filtering by
columns gives an $E^2$ spectral sequence calculating the homology of
the total complex associated to a bisimplicial abelian group
consisting of what we get by first taking vertical homology and then
taking horizontal homology.  In the case of the bisimplicial abelian
group we have in Equation (\ref{def:bicomp}), the vertical $q$th homology of
the columns will be the $q$th homology with respect to $\sum_{i=0}^n (-1)^i
 d^F_i$ of the complex  
 \[ \bigotimes_{b\in B_m} \L_F^R(A) \]
 and this is isomorphic to $\pi_q\left(\bigotimes_{b\in B_m} \L_F^R(A)\right)$. 
Since we assumed that $\pi_*( \L_F^R(A))$ is flat over $R$, we obtain  
\[ \pi_q\left(\bigotimes_{b\in B_m} \L_F^R(A)\right) \cong
  (\bigotimes_{b\in B_m} \pi_*( \L_F^R(A)))_q. \]
  Here, the subscript $q$ denotes the degree $q$ part of the graded
  abelian group $\bigotimes_{b\in B_m} \pi_*( \L_F^R(A))$.  
  
Moreover, the effect of the horizontal boundary map on
$\bigotimes_{b\in B_m} \pi_*( \L_F^R(A))$ is the boundary of the twisted Loday
construction, with the action of
$G$ on the graded-commutative $R$-algebra $\pi_*( \L_F^R(A))$  induced
by  that  of $G$ on the commutative simplicial $R$-algebra
$\L_F^R(A)$. As the boundary map preserves internal degree, we get the
desired spectral sequence. 
\end{proof}

\subsection{Norms and finite coverings of $S^1$}
The connected $n$-fold cover of $S^1$ given by the degree $n$ map can
be made into a TCP as follows. Let $B = 
S^1$ be the standard simplicial circle and $C_n = \langle \gamma
\rangle$ be the cyclic group of order $n$ with 
generator $\gamma$. The twisting function $\tau \colon S^1_q \to C_n$ sends
the non-degenerate simplex in 
$S^1_1$ to $\gamma$ and is then determined by Equation
\eqref{tauconst}. Let $F = 
C_n$, viewed as a constant simplicial set, and let $C_n$ act on $F$
from the left. Then $E(\tau) = F 
\times_{\tau} B$ is in fact another 
simplicial model of $S^1$ with $n$ non-degenerate
$1$-simplices. Therefore,
\[ \L^R_{E(\tau)}(A) \simeq \L^R_{S^1}(A) \quad \text{ and } \quad 
  \pi_*(\L^R_{E(\tau)}(A)) \cong \HH_*^R(A)\]
for every commutative $R$-algebra $A$. In this case,
$\mathcal{L}^R_F A = A^{\otimes_R n}$ is the constant commutative
simplicial $R$-algebra, with the $C_n$-action given by 
\begin{equation*}
\gamma(a_1 \otimes \cdots \otimes a_n)= a_n \otimes a_1 \otimes \cdots
\otimes a_{n-1}. 
\end{equation*}
As $\mathcal{L}^R_F(A)$ is a constant simplicial object, we obtain
that
\[ \pi_*\L_F^R(A) \cong \begin{cases}A^{\otimes_R n}, & *=0, \\
    0, & * > 0.\end{cases}\]
If $A$ is flat over $R$, the spectral sequence of Equation
\eqref{eq:tcpspsq} is  
\begin{equation*} 
  E^2_{p,q} = \pi_p(\mathcal{L}^R_{S^1}(A^{\otimes_R n})^\tau)_q
  \Rightarrow \pi_{p+q} \mathcal{L}^R_{E(\tau)}(A) \cong \HH_{p+q}^R(A). 
\end{equation*}
But here, the spectral sequence is concentrated in $q$-degree zero,
and hence it collapses, yielding
\[ \pi_p(\mathcal{L}^R_{S^1}(A^{\otimes_R n})^\tau) \cong \HH^R_{p}(A).\]

With Proposition \ref{LodayconstructionTCP} we can identify
$\L_{E(\tau)}^S(A)$ if $A$ is a commutative ring spectrum and we
recover the known result (see for instance \cite[p.~2150]{abghl}) that 
\begin{equation} \label{eq:thhnorm}
\THH_{C_n}(N_e^{C_n} A) \simeq \THH(A).
\end{equation}
Here, $\THH_{C_n}(A) = N_{C_n}^{S^1}(A)$ is the $C_n$-relative $\THH$ defined in
\cite[Definition 8.2]{abghl},  where $N_e^{C_n}A$ is 
the Hill-Hopkins-Ravanel norm. See also \cite[Definition
2.0.1]{aghkk}. The identification in \eqref{eq:thhnorm} is an instance
of 
the transitivity of the norm: 
$N_{C_n}^{S^1}N_e^{C_n}A \simeq  N_e^{S^1}A$.

\subsection{The case of the Klein bottle}
For the Klein bottle we compute the homotopy groups of the Loday
construction of the polynomial algebra $k[x]$ for a field $k$ using
our TCP spectral sequence and we confirm our answer using the following
pushout argument. We assume that the characteristic of $k$ is not $2$, so $2$
is invertible in $k$. 

Note that the Klein bottle can be represented as a homotopy pushout
$K\!\ell \simeq (S^1 \vee
S^1) \cup^h_{S^1} D^2$.  Since the Loday construction converts
homotopy pushouts of 
simplicial sets into homotopy pushouts of commutative algebra spectra,
we obtain
\[\L^{k}_{K\!\ell}(k[x]) \simeq \L^k_{S^1 \vee
  S^1}(k[x]) \wedge^L_{\L^k_{S^1}(k[x])} \L^k_{D^2}(k[x]). 
\]
Homotopy invariance of the Loday construction yields that
$\pi_*\L^k_{D^2}(k[x]) \cong k[x]$, and as
$\pi_*\L^k_{S^1}(k[x]) = \HH_*^k(k[x])$ is well known to be isomorphic to
$k[x] \otimes \Lambda(\varepsilon x)$ as a graded commutative
$k$-algebra, we get that  
\[ \pi_*\L^k_{S^1 \vee S^1}(k[x]) \cong \pi_*\L^k_{S^1}(k[x]) \otimes_{k[x]}
                                   \pi_*\L^k_{S^1}(k[x]) \cong k[x] \otimes
                                   \Lambda(\varepsilon x_a, \varepsilon x_b) \]
where the indices $a$ and $ b$ allow us to distinguish between the
generators emerging 
from each of the circles $S^1_a \vee S^1_b$.  Let  $S^1_c$
represent the circle along which we glue the disk, and call the corresponding generator in dimension one for the Loday construction over it  
$\varepsilon x_c$. Let $ S^1_a$
denote the circle that $S^1_c$ will go twice around in the same
direction and $S^1_b$ denote the circle that it will go around in
opposite directions.  So we have a projection $K\!\ell\to S^1_b$. 

We can calculate $\pi_*\L^k_{K\!\ell}(k[x])$ with a Tor spectral sequence
whose  $E^2$-page is  

\begin{align}
    E^2_{*,*} & = \Tor^{\pi_* \L^k_{S^1}(k[x])}_{*,*}\bigg(\pi_* \L^k_{S^1 \vee
                S^1} (k[x]), \pi_* \L^k_{D^2} (k[x])\bigg)\label{spseqKl} \\ 
    &\cong \Tor^{k[x] \otimes \Lambda(\varepsilon x_c)}_{*,*}(k[x] \otimes
      \Lambda(\varepsilon x_a, \varepsilon x_b), k[x]). \nonumber
\end{align}

We need to understand the $\pi_* \L^k_{S^1}(k[x])$-module structure on
$k[x] \otimes \Lambda(\varepsilon
x_a, \varepsilon x_b)$, so we need to understand  the map
\[k[x] \otimes \Lambda(\varepsilon x_c) \to k[x] \otimes \Lambda(\varepsilon
x_a, \varepsilon x_b). \]
Since $k[x]$ in both cases is the image of the Loday construction on a
point, we know that $x$ on the left maps to $x$ on the right. 
If we map  $S^1_c$ to $S^1_a \vee S^1_b$ and then collapse $S^1_a$ to a
point, we end up with a map 
$S^1_c\to S^1_b$ that is contractible, so if we only look at the
$\varepsilon x_b$ part of the image of $\varepsilon x_c$ in
$\Lambda(\varepsilon 
x_a, \varepsilon x_b)$ (that is, if we augment $\varepsilon x_a$ to
zero) we get zero. 

We deduce that 
\begin{align*}
\Tor^{k[x] \otimes \Lambda(\varepsilon x_c)}_{*,*} &(k[x] \otimes
      \Lambda(\varepsilon x_a, \varepsilon x_b), k[x])  \\
      & \cong
      \Tor^{k[x]}_{*,*} (k[x], k[x])\otimes 
      \Tor^{ \Lambda(\varepsilon x_c)}_{*,*}(
      \Lambda(\varepsilon x_a), k)  \otimes \Tor^k_{*,*}
        (\Lambda(\varepsilon x_b), k) \\ 
  &\cong k[x]\otimes 
      \Tor^{ \Lambda(\varepsilon x_c)}_{*,*}(
      \Lambda(\varepsilon x_a), k) \otimes      \Lambda(\varepsilon x_b). 
      \end{align*}    
 In order to calculate $\Tor^{ \Lambda(\varepsilon x_c)}_{*,*}(\Lambda(\varepsilon
 x_a), k)$,  we map  $S^1_c$ to $S^1_a \vee 
      S^1_b$ and then collapse $S^1_b$ to a point.  This gives a map
      $S^1_c\to S^1_a$ that is homotopic to  the double cover of the
      circle as depicted below. 
We consider elements of $\L_{S^1_c}(k[x])$, which we think of as built
on the top circle, and of $\L_{S^1_a}(k[x])$, which we think of as
built on the bottom circle, and write them as sums of tensor monomials
of ring elements with subscripts indicating the simplex each ring
element lies over.  Under this map, we have 
\begin{minipage}{.25\textwidth}
\setlength{\unitlength}{1cm}
\thicklines
\begin{picture}(5,8)
\put(2,2){\circle{2}}
\put(2,3.9){\vector(0,-1){1}}
\put(2,5){\circle{2}}
\put(2,5.7){\circle*{0.1}}
\put(1.9,5.8){$v_1$}
\put(.7,5){$\alpha_1$}
\put(2,4.3){\circle*{0.1}}
\put(2,4){$v_0$}
\put(3,5){$\alpha_0$}
\put(2,1.3){\circle*{0.1}}
\put(2,1){$v$}
\put(2.9,2){$\alpha$}
\put(2.1,1.31){\vector(-1,0){.1}}
\put(1.3,4.9){\vector(0,-1){.1}}
\put(2.7,5){\vector(0,1){.1}}
\end{picture}
 
\end{minipage}
\begin{minipage}{.75\textwidth}
\begin{align*}
    x_{\alpha_0} := 1_{s_0v_0} \otimes 1_{s_0v_1} \otimes x_{\alpha_0}
  \otimes 1_{\alpha_1} \mapsto 1_{s_0v} \otimes x_{\alpha} \\ 
    x_{\alpha_1} := 1_{s_0v_0} \otimes 1_{s_0v_1} \otimes 1_{\alpha_0}
  \otimes x_{\alpha_1} \mapsto 1_{s_0v} \otimes x_{\alpha} 
\end{align*}
Then $d_ 0 - d_1$ maps these elements to the following:
\begin{align*}
  x_{\alpha_0}
  \mapsto 1_{v_0} \otimes x_{v_1} - x_{v_0}
  \otimes 1_{v_1}\\ 
  x_{\alpha_1}
  \mapsto x_{v_0} \otimes 1_{v_1} - 1_{v_0}
  \otimes x_{v_1} 
\end{align*}
Note that the sum of the images under $d_0-d_1$ is zero, and so $
x_{\alpha_0}+ x_{\alpha_1}$ is a cycle with one copy of $x$ 
in simplicial degree 1, which is what $\varepsilon x_c$ should be.
Monomials that put the copy of $x$ over $s_0v_i$ are the image under
$d_0-d_1+d_2$ of monomials that put one copy of $x$ over $s_0^2 v_i$,
so do not contribute to the homology, and all cycles not involving
those and involving only one copy of $x$ are multiples of  $
x_{\alpha_0}+ x_{\alpha_1}$, and so $ x_{\alpha_0}+ x_{\alpha_1}$
represents $\varepsilon x_c$.  But we know that $\varepsilon x_a$ is
represented by $1_{s_0v} \otimes x_{\alpha}$, 
so we get that that $\varepsilon x_c \mapsto 2 \varepsilon x_a$. 
\end{minipage}
\medskip
We take the standard resolution of  $k$ as a
$\Lambda(\varepsilon x_c)$-module:

\[ \xymatrix{\ldots \ar[rr]^{\cdot \varepsilon x_c} & &
    \Sigma\Lambda(\varepsilon x_c) 
  \ar[rr]^{\cdot \varepsilon x_c} & &\Lambda(\varepsilon
 x_c) }\]

Since we saw above that $\varepsilon x_c \mapsto 2\varepsilon
x_a$, tensoring  $(-)\otimes_{\Lambda(\varepsilon
  x_c)} \Lambda(\varepsilon x_a)$ yields
\[ \xymatrix{\ldots \ar[rr]^{\cdot 2\varepsilon x_a} & &
    \Sigma\Lambda(\varepsilon x_a) 
  \ar[rr]^{\cdot 2\varepsilon x_a} & &\Lambda(\varepsilon
 x_a) }\]
Since we assume that $2$ is invertible in $k$, we get that
$\Tor^{\Lambda(\varepsilon x_c)}_*(\Lambda(\varepsilon x_a), k) \cong
k$,  
and so when $2$ is invertible in $k$, the spectral sequence in
Equation (\ref{spseqKl}) has the form
\[
    E^2_{*,*} \cong k[x] \otimes k \otimes \Lambda(\varepsilon x_b)
                \cong k[x] \otimes \Lambda(\varepsilon x_b),  
\]
and therefore also collapses for degree reasons, yielding 
\[
\pi_*\L^k_{K\!\ell}(k[x]) \cong k[x] \otimes \Lambda(\varepsilon x).
\]
\begin{rem}
In fact we have shown that $\pi_*\L^k_{K\!\ell}(k[x]) \cong \pi_*
\L^k_{S^1}(k[x])$ and that the projection $K\!\ell \to S^1_b$ induces this
isomorphism. 
\end{rem}

Now we want to get the same result using our TCP spectral
sequence for $A=k[x]$.  We will use the following simplicial model for
the Klein bottle: 
\[K\!\ell= (I \times S^1)/(0,t) \sim (1, \text{flip}(t)) \]
where flip is the
reflection of the circle about the $y$-axis.  If we use the same model
of the circle with two vertices and two edges that we used in the
double cover picture above but we reverse the orientation on
$\alpha_0$ so that both edges go top to bottom, this is a simplicial
map preserving $v_0$ and $v_1$ and exchanging the $\alpha_i$. 

The flip map induces a map on $\pi_*(\L_{S^1}^k(k[x]) \cong  k[x]
\otimes \Lambda(\varepsilon x)$ sending $x\mapsto x$ and $\varepsilon
x\mapsto- \varepsilon x$.  The fact that $x\mapsto x$ comes from the
fact that it is the image of the Loday construction over a point.
Using the same notation and argument as before, with the different
orientation on $\alpha_0$, $\varepsilon x$ can be represented by
$x_{\alpha_0}-x_{\alpha_1}$, so exchanging the $\alpha_i$ sends
$\varepsilon x$ to $- \varepsilon x$.  

The nontrivial twist $\tau \colon S^1 \to C_2=\langle\gamma\rangle$ maps
the non-degenerate $1$-cell $\alpha \in S^1_1$ to $\gamma$ and is then
determined by  Equation \eqref{tauconst}, yielding
\begin{equation} \label{eq:gammaaction}
  d_0(a_0 \otimes a_1 \otimes \ldots \otimes a_n) = a_0 \cdot \gamma
  a_1 \otimes a_2 \otimes \ldots \otimes a_n.
\end{equation} 
The TCP spectral sequence \eqref{eq:tcpspsq} in this case takes the form
\[E^2_{p,q} = \pi_p \left(\left( \L^k_{S^1}\bigg(\pi_*
      (\L^k_{S^1}(k[x]))^\tau\bigg)\right)_q\right) 
  \implies \pi_{p+q}\L^k_{K\!\ell}(k[x])\] 
and since $ \pi_* \L^k_{S^1}(k[x]) \cong k[x] \otimes \Lambda
(\varepsilon x)$, 
\[ E^2_{p,q} =\pi_p \left(\left(\L^k_{S^1}\bigg( k[x] \otimes
      \Lambda (\varepsilon x)^\tau\bigg)\right)_q\right), 
\] 
which is the $p$th homotopy group of the simplicial $k$-vector space
whose $p$-simplices are 
\[\left(\L_{S^1_p}^k(k[x] \otimes \Lambda (\varepsilon x)^\tau)\right)_q.\]

For each $p$, $\L_{S^1_p}^k(k[x] \otimes \Lambda (\varepsilon
x))\simeq \L_{S^1_p}^k(k[x])\otimes _k \L_{S^1_p}^k(\Lambda
(\varepsilon x))$, and so $\L_{S^1}^k(k[x] \otimes \Lambda (\varepsilon
x)^\tau)\simeq \L_{S^1}^k(k[x])\otimes _k \L_{S^1}^k(\Lambda
(\varepsilon x)^\tau)$.
We can think of this tensor product of simplicial $k$-algebras as the
diagonal of a bisimplicial abelian group, and again by 
\cite[Theorem 2.4 of Section IV.2.2]{gj}  the total
complex of a bisimplicial abelian group with the alternating sums of
the vertical and the horizontal face maps is chain homotopy equivalent
to the usual chain complex associated to the diagonal of that
bisimplicial abelian group. 
But in this case of a tensor product, the total complex was obtained
by tensoring together two complexes, and since we are working over a
field its homology is the tensor product of the homology of the two
complexes, so 
\begin{align*}
\pi_* \left(\left(\L^k_{S^1}\bigg(k[x] \otimes
  \Lambda (\varepsilon x)^\tau\bigg)\right)_*\right) &\cong 
\pi_* \left(\left( \L^k_{S^1} ( k[x]) \right)_*\right) 
\otimes\pi_* \left(\left(\L^k_{S^1} ( \Lambda
(\varepsilon x)^\tau)\right)_*\right).
\end{align*}
The first factor is just the Hochschild homology of $k[x]$. It sits
in the $0$th row of the $E^2$ term since $x$ has internal degree zero,
and gives us $\pi_*(\L^k_{S^1} (k[x]) \cong k[x] \otimes
\Lambda(\varepsilon x))$ concentrated in positions $(0,0)$ and $(1,0)$.
All spectral sequence differentials vanish on it for degree reasons,
and so it will just contribute $k[x] \otimes \Lambda(\varepsilon x)$
to the $E^\infty$ term. 

The second factor in the $E^2$ term is the twisted Hochschild homology for
$\Lambda(\varepsilon x)$.  To calculate it, we can use the normalized
chain complex and therefore we only have to  
consider non-degenerate elements, which means that we only have two 
elements to take into account in any given simplicial degree: 

\begin{center}
\begin{tabular}{|c|c|c|c|c}
\hline
$p$-degree & $0$ & $1$ & $2$ & \ldots \\
\hline
       & $1$ & $1 \otimes \varepsilon x$ & $1 \otimes \varepsilon x
                                           \otimes \varepsilon x$ & 
                                               $\ldots$ \\ 
  \hline
 &  $\varepsilon x$ & $\varepsilon x \otimes \varepsilon x$ &
     $\varepsilon x \otimes \varepsilon x \otimes \varepsilon x$ &
                                                                   $\ldots$
  \\ \hline 
\end{tabular}
\end{center}

\bigskip

Elements of the form $\varepsilon x \otimes \ldots \otimes \varepsilon
x$ will map to zero under the Hochschild boundary 
map.  We need to consider the odd and even cases of
differentials on elements of the form $1 \otimes \varepsilon x \otimes
\ldots \otimes \varepsilon x$.  The $ d_i$ maps in the
twisted  and untwisted Hochschild complex are all the same except
$ d_0$, which 
incorporates the twisting action of $\tau$. Therefore we have 
\begin{align*}
    d(1 \otimes (\varepsilon x)^{\otimes 2k}) &= -\varepsilon x^{\otimes 2k}
                                             + (-1)^{2k} (-1) \varepsilon
                                             x^{\otimes 2k} = -2
                                             \varepsilon x^{\otimes 2k}\\ 
    d(1 \otimes (\varepsilon x)^{\otimes 2k+1}) &= -\varepsilon x^{\otimes
                                               2k+1} + (-1)^{2k+1} (1)
                                               \varepsilon x^{\otimes
                                               2k+1} = -2 \varepsilon
                                               x^{\otimes 2k+1}.  
\end{align*}
Here, the first $-1$ comes from the 
  $\gamma$ action on $\epsilon x$ as in \eqref{eq:gammaaction} and the
  extra $\pm 1$ in  
  brackets come from  passing the one-dimensional
$\varepsilon x$ past an odd or an even number of copies of itself.
Since 
we are assuming that $2$ is invertible in $k$, we get that the second
part of the $E^2$ term has only $k$ left in degree $0$. 
So, if $2$ is invertible in $k$, then the entire $E^2$ term is just $k[x]
\otimes \Lambda(\varepsilon x)$ in the $0$th row, and the TCP spectral
sequence collapses and confirms that  
\[\pi_*\L^k_{K\!\ell}(k[x]) \cong k[x] \otimes \Lambda(\varepsilon x).\]
\section{Hopf algebras in spectra}
\label{sec:hopf}

We start by describing what we mean by the notion of a commutative
Hopf algebra in the $\infty$-category of spectra, $\Sp$.
We consider the $\infty$-category $\CAlg$ of 
$E_\infty$-ring spectra.

\begin{defn}
A \emph{commutative Hopf algebra spectrum} is a cogroup object in $\CAlg$. 
\end{defn}

Hopf algebra spectra are fairly rare, so let us list some important examples. 
 
\begin{ex}
If $G$ is a topological abelian group, then the spherical group ring
$S[G] = \Sigma^\infty_+ G$ equipped with the product induced by the
product in $G$, the coproduct induced by the diagonal map $G\to G\times G$,
and the antipodal map induced by the inverse map from $G$ to $G$ is a
commutative Hopf algebra 
spectrum. This follows from the fact that the suspension spectrum 
functor $\Sigma^\infty_+ \colon \cS \to \Sp$ is a strong symmetric
monoidal functor. Here $\mathcal{S}$ denotes the $\infty$-category of spaces.
\end{ex}

\begin{ex}
If $A$ is an ordinary commutative Hopf algebra over a commutative ring
$k$ and $A$ is flat as a $k$-module then the
Eilenberg-Mac Lane spectrum $HA$ is a commutative Hopf
algebra spectrum over $Hk$ because the canonical map 
\[
HA \wedge_{Hk} HA \to H(A \otimes_k A)
\]
is an equivalence. 
\end{ex}

We use the fact that the category of commutative ring spectra is
tensored over unpointed topological spaces and simplicial sets in a
compatible way \cite[VII, \S 2, \S 3]{ekmm}. 

If $\mathcal{U}$ denotes the category of unbased (compactly generated
weak Hausdorff) spaces and $X \in \mathcal{U}$, then for every pair of
commutative ring spectra $A$ and $B$ there is a homeomorphism of
mapping spaces (\cite[VII,
Theorem 2.9]{ekmm}) 
\begin{equation} \label{eq:adj1}
  \mathcal{C}_S(X \otimes A, B)  \cong \mathcal{U}(X,
  \mathcal{C}_S(A,B)).\end{equation}
Here, $\mathcal{C}_S$ denotes the (ordinary) category of commutative
ring spectra in the sense of \cite{ekmm}.
By \cite[Corollary 4.4.4.9]{HTT}, \eqref{eq:adj1} corresponds to an  
equivalence of mapping spaces of $\infty$-categories
\begin{equation} \label{eq:adj1a}
 \CAlg(X \otimes A, B)  \simeq
                               \mathcal{S}(X, \CAlg(A,B)).  
\end{equation}
See also \cite[\S 2]{rsv} for a detailed account on tensors in
$\infty$-categories.

If we consider a commutative
Hopf algebra spectrum $\niceH$, then the space of maps
$\mathcal{C}_S(\niceH,B)$ has a basepoint: the composition of the counit
map $\niceH \ra S$ followed by the unit map $S \ra B$ is a map of
commutative ring spectra. The functor that takes an unbased space $X$
to the topological sum of $X$ with a point $+$ is left adjoint
to the forgetful functor the category of pointed spaces,
$\text{Top}_*$,  to spaces, so
we obtain a homeomorphism 
\begin{equation} \label{eq:adj2}
  \mathcal{U}(X, \mathcal{C}_S(\niceH,B)) \cong \text{Top}_*(X_+,
  \mathcal{C}_S(\niceH,B))\end{equation}
and correspondingly, an equivalence in the context of
$\infty$-categories 
\begin{equation} \label{eq:adj2a}
  \mathcal{S}(X, \CAlg(\niceH,B)) \simeq \mathcal{S}_*(X_+,
  \CAlg(\niceH,B)).\end{equation}
For path-connected spaces $Z$, May showed that the free
$E_n$-space on $Z$, $C_n(Z)$, is equivalent to $\Omega^n\Sigma^nZ$
\cite[Theorem 6.1]{gils}.   
Segal extended this result to spaces that are not necessarily
connected. He showed that for well-based spaces $Z$  there
is a model of the free $E_1$-space, $C'_1(Z)$, as follows: The spaces
$C_1(Z)$ and $C'_1(Z)$ are homotopy equivalent, $C_1'(Z)$ is a monoid, its
classifying space $BC'_1(Z)$ is 
equivalent to $\Sigma(Z)$ \cite[Theorem
2]{segal}, and thus, $C'_1(Z) \ra \Omega BC'_1(Z)$ is a group
completion. We can apply this result to $Z=X_+$ because $X_+$
is well-based, thus $BC'_1(X_+) \simeq \Sigma(X_+)$. Note that $\Omega
BC'_1(X_+) \simeq \Omega\Sigma(X_+)$.

Nikolaus gives an overview about group completions in the
context of $\infty$-categories \cite{nikolaus}. He shows that for
every $E_1$-monoid $M$, the map $M \ra \Omega BM$ gives rise to a
localization functor of $\infty$-cateories in the sense of
\cite[Definition 5.2.7.2]{HTT}, such that the local objects are grouplike
$E_1$-spaces. In particular, there is a homotopy equivalence of
mapping spaces \cite[Proposition 5.2.7.4]{HTT}
\[ \text{Map}_{E_1}(\Omega BC'_1(X_+), Y) \simeq
  \text{Map}_{E_1}(C'_1(X_+), Y)\]
if $Y$ is a grouplike $E_1$-space. Here,  $E_1$ denotes the
$\infty$-category of $E_1$-spaces.

If $\niceH$ is a commutative
Hopf-algebra, then the space $\CAlg(\niceH,B)$ is a grouplike
$E_1$-space. Therefore, by using Equations \eqref{eq:adj1a} and
\eqref{eq:adj2a}, we obtain a chain of homotopy equivalences 
\begin{align*} 
  \CAlg(X \otimes \niceH, B) &  \simeq
                               \mathcal{S}(X, \CAlg(\niceH,B)) \\
                             & \simeq \mathcal{S}_*(X_+, \CAlg(\niceH,B)) \\
  & \simeq \text{Map}_{E_1}(C'_1(X_+), \CAlg(\niceH,B)) \\
  & \simeq \text{Map}_{E_1}(\Omega
    BC'_1(X_+), \CAlg(\niceH,B)) \\
  & \simeq \text{Map}_{E_1}(\Omega\Sigma(X_+),
\CAlg(\niceH,B)).  
\end{align*}
If $\Sigma(X_+) \simeq \Sigma(Y_+)$ is an equivalence of pointed spaces, then
$\Omega\Sigma(X_+) \simeq \Omega\Sigma(Y_+)$ as grouplike $E_1$-spaces
and therefore we get a homotopy equivalence 
\[ \CAlg(X \otimes \niceH, B) \simeq \CAlg(Y \otimes \niceH, B).\]

Applying the Yoneda Embedding to the above equivalence yields the
following result: 
\begin{thm} \label{thm:bravenewhopf}
If $\niceH$ is a commutative Hopf algebra spectrum and if $\Sigma(X_+)
\simeq \Sigma(Y_+)$ is an equivalence of pointed spaces, then 
there is an equivalence $X \otimes \niceH \simeq Y \otimes \niceH$ in $\CAlg$. 
\end{thm}

\begin{rem}
If $X$ is a pointed simplicial set, then the suspension $\Sigma(X_+)$
is equivalent to $\Sigma(X) \vee S^1$. Therefore, if $X$ and $Y$ are
pointed simplicial sets, such that $\Sigma(X) \simeq \Sigma(Y)$ as
pointed simplicial sets, then
we also obtain an equivalence between $\Sigma(X_+)$ and $\Sigma(Y_+)$. 
\end{rem}

Segal's result also works for larger $n$ than $1$. 
If two spaces are equivalent after an $n$-fold suspension, then an
$E_n$-coalgebra structure on a Hopf algebra is needed for the Loday
construction to be equivalent on these two spaces. There are indeed
interesting spaces that are not equivalent after just one suspension,
but that need iterated suspensions to become equivalent: 
\begin{itemize}
  \item
Christoph Schaper \cite[Theorem 3]{Schaper} shows that for affine 
arrangements $\mathcal{A}$ 
one needs at least a $(\tau_\mathcal{A}+2)$-fold  
suspension in order to get a homotopy type that only
depends on the poset structure of the arrangement. Here,
$\tau_\mathcal{A}$ is a number that depends on the poset data of the
arrangement, namely the intersection poset and the dimension function. 
\item
For homology spheres, the double suspension theorem of 
James W. Cannon and Robert D. Edwards \cite[Theorem in \S 11]{cannon}
states that the double suspension $\Sigma^2 M$ of any $n$-dimensional
homology sphere $M$ is homeomorphic to $S^{n+2}$. Here, a single
suspension does \emph{not} suffice unless $M$ is an actual sphere. 
\end{itemize}

\section{Truncated polynomial algebras}
\label{sec:trunc}
One way of showing that a commutative $R$-algebra spectrum $A$ is not
multiplicatively or linearly stable is to prove that the homotopy
groups of the Loday 
construction $\L_{T^n}^R(A)$ differ from those of $\L_{\bigvee_{k=1}^n
  \bigvee_{\binom{n}{k}} S^k}^R(A)$, as in \cite{DT}. Here, we write
$\bigvee_{\binom{n}{k}} S^k$ for the $\binom{n}{k}$-fold $\vee$-sum of
$S^k$. 
Indeed, there is 
a homotopy equivalence 
\[ \Sigma(T^n) \simeq \Sigma(\bigvee_{k=1}^n\bigvee_{\binom{n}{k}}
  S^k).\]
If $A$ is augmented over $R$, then for proving that $R \ra A$ is not
multiplicatively or additively stable, it suffices to show that
\[ \L_{T^n}^R(A; R) \not\simeq \L_{\bigvee_{k=1}^n
  \bigvee_{\binom{n}{k}} S^k}^R(A; R). \] 
See \cite[\S 2]{lr} for details and background on different notions of
stability. 

In the following we restrict our attention to Eilenberg-Mac Lane
spectra of commutative rings and we will use this strategy to show
that none of the commutative  $\Q$-algebras $\Q[t]/t^m$ for $m \geq 2$
can be multiplicatively stable. We later generalize this to quotients
of the form $\Q[t]/q(t)$ where $q(t)$ is a polynomial without constant
term and to integral and mod-$p$ results.

Pirashvili determined higher order Hochschild homology of truncated 
polynomial algebras of the form $k[x]/x^{r+1}$ additively when $k$ is a field of
characteristic zero \cite[Section 5.4]{pira-hodge} in the case of odd
spheres. A direct adaptation of the methods of \cite[Theorem 8.8]{blprz}
together with the flowchart from \cite[Proposition 2.1]{bhlprz}
yields the higher order Hochschild homology with reduced coefficients
for all spheres. See also \cite[Lemma 3.4]{DT}. 

\begin{prop} \label{prop:spherestruncated}
For all $m \geq 2$ and $n \geq 1$ 
\[ \HH_*^{[n], \Q}(\Q[t]/t^m; \Q) \cong \begin{cases}
\Lambda_{\Q}(x_n) \otimes \Q[y_{n+1}], & \text{ if } n \text{ is odd},
\\
\Q[x_n] \otimes \Lambda_{\Q}(y_{n+1}), & \text{ if } n \text{ is even.}
  \end{cases} \]
\end{prop}
In both cases Hochschild homology of order $n$ is a free graded
commutative $\Q$-algebra on two generators in degrees $n$ and
$n+1$, respectively, and the result does not depend on $m$.

We will determine for which $m$ and $n$ we get a decomposition of the form 
\begin{equation} \label{eq:decomposestori}
\pi_*\L_{T^n}^\Q(\Q[t]/t^m; \Q) \cong
\pi_*\L_{\bigvee_{k=1}^n\bigvee_{\binom{n}{k}} S^k}^\Q(\Q[t]/t^m;
\Q). 
\end{equation}
Note that the right-hand side is isomorphic to 
\[\bigotimes_{k=1}^n
\bigotimes_{\binom{n}{k}} \pi_*\L_{S^k}^\Q(\Q[t]/t^m; \Q)\]
where all unadorned tensor products are formed over $\Q$. Thus, if we
have a decomposition as in \eqref{eq:decomposestori}, then we can
read off the homotopy groups of $\pi_*\L_{T^n}^\Q(\Q[t]/t^m; \Q)$ with
the help of Proposition \ref{prop:spherestruncated}. 

Expressing $\Q[t]/t^m$ as the pushout of the diagram 
\[ \xymatrix{\Q[t]\ar[rr]^{t \mapsto t^m} \ar[d]_{t \mapsto 0}& & \Q[t] \\
\Q  } \]
allows us to express the Loday construction for $\Q[t]/t^m$, now
viewed as a commutative $H\Q$-algebra spectrum,  as the
homotopy pushout of the diagram
\[\xymatrix{\L^{H\Q}_{T^n}(H\Q[t]; H\Q) \ar[rr]^{t \mapsto t^m} \ar[d]_{t
      \mapsto 0}& & \L^{H\Q}_{T^n}(H\Q[t]; H\Q)  \\
H\Q} \]
and so
\[ \L_{T^n}^{H\Q}(H\Q[t]/t^m; H\Q) \simeq\L^{H\Q}_{T^n}(H\Q[t]; H\Q)
  \wedge^L_{\L^{H\Q}_{T^n}(H\Q[t]; H\Q)} H\Q. \] 
As $\Q[t]$ is smooth over $\Q$, $\L^{H\Q}_{T^n}(H\Q[t]; H\Q)$ is stable \cite[Example
2.6]{DT}.  So we can write 
\[ \L^{H\Q}_{T^n}(H\Q[t]; H\Q) \simeq
  \L^{H\Q}_{\bigvee_{k=1}^n\bigvee_{\binom{n}{k}} S^k}(H\Q[t]; H\Q). \]
Again,  we obtain an isomorphism
\[ \pi_*\L^\Q_{\bigvee_{k=1}^n\bigvee_{\binom{n}{k}}
  S^k}(\Q[t]; \Q) \cong \bigotimes_{k=1}^n \bigotimes_{\binom{n}{k}}
\HH_*^{[k], \Q}(\Q[t]; \Q)\]
and with the help of \cite[Proposition 2.1]{bhlprz} we can identify
the terms as follows: 
\[ \HH_*^{[k], \Q}(\Q[t]; \Q) \cong \begin{cases}
    \Q[x_k], & \text{ if } k \text{ is even }, \\
    \Lambda_\Q(x_k), & \text{ if } k \text{ is odd. }   \end{cases} \]

\begin{lem}
  There is an isomorphism of graded commutative $\Q$-algebras
  \[  \pi_*\L^\Q_{\bigvee_{k=1}^n\bigvee_{\binom{n}{k}}
  S^k}(\Q[t]/t^m; \Q) \cong \pi_*\L^\Q_{T^n}(\Q[t]; \Q) \otimes
\Tor_{*}^{\pi_*\L^\Q_{T^n}(\Q[t]; \Q)}(\Q, \Q).\]  
\end{lem}
\begin{proof}
  We already know that
   \begin{equation} \label{eq:pinchedloday}
     \pi_*\L^\Q_{\bigvee_{k=1}^n\bigvee_{\binom{n}{k}} 
  S^k}(\Q[t]/t^m; \Q) \cong
\bigotimes_{k=1}^n\bigotimes_{\binom{n}{k}} \HH_*^{[k],
  \Q}(\Q[t]/t^m; \Q) \cong
\bigotimes_{k=1}^n\bigotimes_{\binom{n}{k}} gF_\Q(x_k,y_{k+1})
\end{equation}
where $gF_\Q(x_k)$ denotes the free graded commutative
$\Q$-algebra generated by an element $x_k$ in degree $k$ and
$gF_\Q(x_k,y_{k+1})$ denotes the free graded commutative 
$\Q$-algebra generated by an element $x_k$ in degree $k$ and an
element $y_{k+1}$ in degree $k+1$. 

As $\pi_*\L^\Q_{T^n}(\Q[t]; \Q) \cong
\bigotimes_{k=1}^n\bigotimes_{\binom{n}{k}} gF_\Q(x_k)$, we obtain
that
\[ \Tor_{*}^{\pi_*\L^\Q_{T^n}(\Q[t]; \Q)}(\Q, \Q) \cong
  \bigotimes_{\ell=2}^{n+1}\bigotimes_{\binom{n}{\ell-1 }}
  gF_\Q(y_\ell)\]  and hence the tensor product of the two gives a
graded commutative $\Q$-algebra isomorphic to \eqref{eq:pinchedloday}
\end{proof}

Let $A_*$ denote the graded commutative $\Q$-algebra
$\pi_*\L^\Q_{T^n}(\Q[t]; \Q)$ and $B_*$ denote
$\pi_*\L^\Q_{T^n}(\Q[t]; \Q)$ viewed as an $A_*$-module via a morphism
of graded commutative $\Q$-algebras $f\colon A_* \ra B_*$.

\begin{lem} \label{lem:modulestructure}
Let $f_1 \colon A_* \ra B_*$ be the morphism $f_1 =  \eta_{B_*} \circ
\varepsilon_{A_*}$ where $\varepsilon_{A_*} \colon A_* \ra \Q$ is the
augmentation that sends all elements of positive degree to zero and
where $\eta_{B_*} \colon \Q \ra B_*$ is the unit  map of $B_*$. Let
$f_2 \colon A_* \ra B_*$ be any map of graded commutative algebras
such that there is an element $x \in A_n$ with $n > 0$ such that
$f_2(x) = w \neq 0$. Let $\Tor_{*,*}^{A_*, f_i}(B_*, \Q)$ denote the
graded 
$\Tor$-groups calculated with respect to the $A_*$-module structure on
$B_*$ given by $f_i$. Then
\[ \dim_\Q (\Tor_{*,*}^{A_*, f_2}(B_*, \Q))_n < \dim_\Q
  (\Tor_{*,*}^{A_*, f_1}(B_*, \Q))_n\]
where $(\Tor_{*,*}^{A_*, f_i}(B_*, \Q))_n =
\bigoplus_{r+s=n}\Tor_{r,s}^{A_*, f_i}(B_*, \Q)$. 
\end{lem}
\begin{proof}   
Let $P_* \ra \Q$ be an $A_*$-free resolution of $\Q$. We want to
choose $P_*$ efficiently, in the following sense: since $\Q$ is 
concentrated in degree zero and $A_0=\Q$,  we can choose $P_0$ to be
$A_*$.  Then we choose $P_1=  \bigoplus_{j \in I_1} \Sigma^{n_j} A_*$
with the minimal possible number of copies of $A_*$ in each suspension
degree, beginning from the bottom (that is, the only reason we add a
new $\Sigma^n A_*$ is if there is a class in $P_{\ell-1}$ that has not
yet been hit by the suspensions of $A_*$ in lower dimensions that we
already have) to guarantee that 
$d_1 \colon \bigoplus_{j \in I_\ell} \Sigma^{n_j} A_0 \to P_0$ is
injective, and moreover  
\[\ker(d_1 \colon  \bigoplus_{j \in I_1} \Sigma^{n_j} A_0 \to P_0)
\subseteq \bigoplus_{j \in I_1 } \Sigma^{n_j} \ker(\epsilon_{A_*}). \]
And of course we need $\Ima(d_1\colon P_1\to P_0) =
\ker(\epsilon_{A_*} \colon A_*\to\Q)$ and similarly for higher
$\ell$. For every $\ell>0$  we choose
$P_\ell$ with 
\[ P_\ell = \bigoplus_{j \in I_\ell} \Sigma^{n_j} A_*\]
so that $d_\ell \colon \bigoplus_{j \in I_\ell} \Sigma^{n_j} A_0 \to
P_{\ell-1}$ is injective and moreover 
\[ \ker(d_\ell \colon \bigoplus_{j \in I_\ell} \Sigma^{n_j} A_0 \to
  P_{\ell-1}) \subseteq \bigoplus_{j \in I_\ell }\Sigma^{n_j}
  \ker(\epsilon_{A_*}). \]
Then we get
\[ \Ima(d_\ell \colon P_\ell\to P_{\ell-1} ) = \ker(d_{\ell-1} \colon
  P_{\ell-1}\to P_{\ell-2})\subseteq \bigoplus_{j \in I_{\ell
      -1}}\Sigma^{n_j} \ker(\epsilon_{A_*}).\] 
The Tor groups we want are   the homology groups of
\[B_* \otimes_{A_*} P_\ell = B_* \otimes_{A_*} \bigoplus_{j \in
    I_\ell} \Sigma^{n_j} A_* \cong \bigoplus_{j \in
    I_\ell} \Sigma^{n_j} B_*\]
with respect to the differential $\id \otimes d$ for either
$A_*$-module structure. 

\smallskip
As $f_1 \colon A_* \ra B_*$ factors through
the augmentation, we claim that the differentials in the chain complex
\[ B_* \otimes_{A_*} P_\bullet\]
with the $A_*$-module structure given by $f_1$ are trivial: they are
of the form $\id \otimes d$ where $d$ is the 
differential of $P_\bullet$. As $d$ sends every $\Sigma^{n_j} 1 \in
\bigoplus_{j \in 
    I_\ell} \Sigma^{n_j} A_*$ to something in 
    $\bigoplus_{j \in I_{\ell -1}}\Sigma^{n_j} \ker(\epsilon_{A_*})$, 
\[ (\id \otimes d)(b\otimes _{A_*} \Sigma^{n_j} 1 )\in \Q\{b\} \otimes_{A_*}
  \bigoplus_{j \in I_{\ell -1}}\Sigma^{n_j} \ker(\epsilon_{A_*})=0 \]
for all $b\in B_*$.  
Hence
\[ \Tor_{\ell,s}^{A_*, f_1}(B_*, \Q) = (\bigoplus_{j \in I_\ell}
  \Sigma^{n_j} B_*)_s = \bigoplus_{j \in I_\ell}
  \Sigma^{n_j} B_{*-s}. \] 
In particular, we have of course $\Tor_{0,s}^{A_*, f_1}(B_*, \Q) =
B_s$ for all $s$.

\smallskip
For the $A_*$-module structure on $B_*$ given by $f_2$ we obtain that
\[ \Tor_{0,*}^{A_*, f_2}(B_*, \Q) = B_* \otimes_{A_*} \Q \]
but here, the tensor product results in a nontrivial quotient of
$B_*$. Recall that we assumed that $f_2(x)=w\neq 0$.  The element $w \otimes 1\in B_* \otimes_{A_*} \Q$ is trivial because the degree of $x$
is positive and hence $\varepsilon_{A_*}(x) = 0$: 
\[ w \otimes 1 = f_2(x) \otimes 1 = 1 \otimes \varepsilon_{A_*}(x) = 1
  \otimes 0 = 0. \]
Therefore, 
\[ \dim_\Q \Tor_{0,n}^{A_*, f_2}(B_*, \Q) < \dim_\Q \Tor_{0,n}^{A_*,
    f_1}(B_*, \Q). \]
The other $\Tor$-terms in total degree $n$ of the form
$\Tor_{r,s}^{A_*, f_2}(B_*, \Q)$ with $r+s=n$ are subquotients of
\[ \bigoplus_{j \in I_r} \Sigma^{n_j} B_{*-s} \] and hence for all
$(r,s)$ with $r+s =n$ and $r > 0$ we obtain 
\[ \dim_\Q \Tor_{r,s}^{A_*, f_2}(B_*, \Q) \leq \dim_\Q
  \Tor_{r,s}^{A_*, f_1}(B_*, \Q). \]
\end{proof}
Note that if $f \colon A_* \ra B_*$ factors through the augmentation $A_* \ra
\Q$ then
\[ \Tor_\ell^{A_*}(B_*, \Q) \cong B_* \otimes \Tor_\ell^{A_*}(\Q,
  \Q). \]

We use Lemma \ref{lem:modulestructure} to prove the following result. 
\begin{thm} \label{thm:notstable}
  Let $n \geq 2$. Then
\[ \dim_\Q\pi_n \L^\Q_{T^n}(\Q[t]/t^n;\Q) < \dim_\Q\pi_n
  \L^\Q_{\bigvee_{k=1}^n\bigvee_{\binom{n}{k}} S^k}(\Q[t]/t^n;\Q). \] 
In particular, for all $n \geq 2$ the pair $(\Q[t]/t^n; \Q)$ is not
stable and $\Q \ra \Q[t]/t^n$ is not multiplicatively stable. 
\end{thm}
The $n=2$ case of  Theorem \ref{thm:notstable} was obtained earlier by
Dundas and Tenti \cite{DT}. 

Before we prove the theorem, we state the following integral version
of it: 
\begin{cor} \label{cor:integral}
  For all $n \geq 2$ the pair $(\Z[t]/t^n; \Z)$ is not stable and $\Z
  \ra \Z[t]/t^n$ is not multiplicatively stable. 
\end{cor}
\begin{proof}[Proof of Corollary \ref{cor:integral}]
If for some $n \geq 2$ the pair $(\Z[t]/t^n; \Z)$ were stable, then in
particular
\[ \pi_* \L^\Z_{T^n}(\Z[t]/t^n;\Z) \cong \pi_*
  \L^\Z_{\bigvee_{k=1}^n\bigvee_{\binom{n}{k}} S^k}(\Z[t]/t^n;\Z). \]
Localizing at $\Z \setminus \{0\}$ would then imply
\[ \pi_* \L^\Q_{T^n}(\Q[t]/t^n;\Q) \cong \pi_*
  \L^\Q_{\bigvee_{k=1}^n\bigvee_{\binom{n}{k}} S^k}(\Q[t]/t^n;\Q) \]
in contradiction to Theorem \ref{thm:notstable}. 
\end{proof}  
\subsection{Proof of Theorem \ref{thm:notstable}}
We prove Theorem
\ref{thm:notstable} by identifying an element in $A_*$ of positive
degree that is sent to a nontrivial element of $B_*$. More precisely,
we will show 
that the map that sends $t$ to $t^n$ sends the indecomposable element in
$\pi_n\L_{S^n}^\Q(\Q[t]; \Q)$
up to a unit to the element
\[ dt_1 \wedge \ldots \wedge
  dt_n \in \pi_n\L_{S^1 \vee \ldots \vee S^1}^\Q(\Q[t]; \Q).  \]
We consider both elements as elements of $\pi_n\L_{T^n}^\Q(\Q[t]; \Q)$
via the inclusions of summands 
\[ \pi_n\L_{S^1 \vee \ldots \vee S^1}^\Q(\Q[t]; \Q) \subset
  \pi_n\L_{T^n}^\Q(\Q[t]; \Q) \supset \pi_n\L_{S^n}^\Q(\Q[t]; \Q).\]

In the following we consider $T^n$ as the diagonal of an $n$-fold
simplicial set where every $([p_1], \ldots, [p_n]) \in (\Delta)^n$
is mapped to $S^1_{p_1} \times \ldots \times S^1_{p_n}$. Then
$\L^\Q_{T^n}(\Q[t];\Q)$ can also be interpreted as the diagonal of an
$n$-fold simplicial $\Q$-vector space with an associated $n$-chain
complex. By abuse of notation we still denote this $n$-chain complex
by $\L^\Q_{T^n}(\Q[t];\Q)$.

We use the following notation concerning the $n$-chain complex
$\L^\Q_{T^n}(\Q[t];\Q)$:  
\begin{itemize}
\item $\mathbf{0}_m=(0,0,\ldots,0)$ and $\mathbf{1}_m=(1,1,\ldots,1)$
  are the vectors containing only $0$ or $1$, respectively, 
  repeated $m$ times. 
\item A vector $\mathbf{V} \in \mathbb{N}^n$ is viewed as a
  \emph{multi-degree} of an 
  element in the $n$-chain complex. 
\item  A vector $\mathbf{v} \in \mathbb{N}^n$ for which $\mathbf{0}_n
  \leq \mathbf{v} \leq 
  \mathbf{V}$ in every entry can be thought of as specifying a
  \emph{coordinate} in the multi-matrix of an 
  element in multi-degree $\mathbf{V}$.
  We call the $i$th entry of a vector $\mathbf{v}\in \mathbb{N}^n$ the
  $i$th  \emph{place} in $\mathbf{v}$. 
  It is always assumed that $\mathbf{V} = \mathbf{1}_n$ if not
  otherwise specified. 
\item Each element of $\L^\Q_{T^n}(\Q[t];\Q)$ in degree $\mathbf{V}=(v_1,
  \ldots, v_n)$ is a 
  multi-matrix of dimension  ${(v_1+1, \ldots, v_n+1)}$ with entries in
  $\Q[t]$ at coordinates $\mathbf{v} \neq \mathbf{0}_n$ and  an entry
  in $\Q$ at  
  coordinate $\mathbf{0}_n$. 
\item $x_{\mathbf{v}}$ for $x \in \Q[t]$ and $\mathbf{v}\in
  \mathbb{N}^n$ is the multi-matrix with term $x$ at 
  coordinate $\mathbf{v}$ and $1$ at other coordinates. We say a term
  is trivial if it is $1$ in all its coordinates. 
\item Therefore $x_{\mathbf{v}} \cdot y_{\mathbf{w}}$ for $x,y \in \Q[t]$ and
  $\mathbf{v}, \mathbf{w} \in 
  \mathbb{N}^n$ is the  product of $x_{\mathbf{v}}$ and
  $y_{\mathbf{w}}$  in degree 
  $\mathbf{V}$ of $\L^\Q_{T^n}(\Q[t],\Q)$ regarded as an $n$-simplicial
  ring. Explicitly, if $\mathbf{v} \neq \mathbf{w}$, it is the
  multi-matrix with $x$ at coordinate 
  $\mathbf{v}$, $y$ at coordinate $\mathbf{w}$, and $1$ elsewhere; if
  $\mathbf{v} = \mathbf{w}$, it 
  is the multi-matrix with $xy$ at coordiante $\mathbf{v}$ and $1$ elsewhere.
\end{itemize}

\newcommand{\ba}{\mathbf{a}}
\newcommand{\bb}{\mathbf{b}}

Suppose that $C_{\bullet}$ is an $n$-chain complex with differentials
$\mathrm{d}_1, \ldots, \mathrm{d}_n$ in the $n$ different directions, then 
the total chain complex $\mathrm{Tot}(C_{\bullet})$ has differential
in component $(v_1, \ldots, v_n)$ given by
\begin{equation*}
\mathrm{d} = \sum_{i=1}^n (-1)^{v_1 + \ldots + v_{i-1}} \mathrm{d}_i.
\end{equation*}
In our case we will have each $\mathrm{d}_i=\sum_{j=0}^{v_i} (-1)^j 
 d_{i,j}$ where $d_{i,j} \colon C_{v_1, \ldots, v_n} \ra C_{v_1,
   \ldots, v_i -1, \ldots v_n}$ is the face map.  We are interested in
 low degrees, especially in 
$\mathbf{1}_n$. Any $v_i=1$ will imply $\mathrm{d}_i=0$  
since the $\mathrm{d}_i$ are cyclic differentials and $\Q[t]$ is commutative.
This allows us to eliminate the $\mathrm{d}_i$ from $\mathrm{d}$.
We have the following three lemmas about homologous classes and tori
of different dimensions: 
\begin{lem}[Split Moving Lemma]
  \label{lem:moving-1}
  Let $\ba,\bb$ be coordinates in degree $\mathbf{1}_{n-1}$ (that is,
  in $2\times 2\times\ldots\times 2$-dimensional matrices).  Then 
\[x_{(\mathbf{a},1)} \cdot y_{(\mathbf{b},1)} \sim
  x_{(\mathbf{a},0)} \cdot y_{(\mathbf{b},1)} + x_{(\mathbf{a},1)}
  \cdot y_{(\mathbf{b},0)}.\] 
\end{lem}
\begin{proof}
Their difference is a boundary of an element of degree
$(\mathbf{1}_{n-1},2)$:    
\[
  \mathrm{d}(x_{(\ba,1)} \cdot y_{(\bb,2)})  =  
   (-1)^{n-1} \mathrm{d}_n (x_{(\ba,1)} \cdot y_{(\bb,2)}) =
  x_{(\ba,0)} \cdot y_{(\bb,1)} - x_{(\ba,1)} \cdot y_{(\bb,1)} +
  x_{(\ba,1)} \cdot y_{(\bb,0)}. 
\]
\end{proof}

For example, when $n=2$, $\ba = 0, \bb = 1$, the difference is 
\begin{equation*}
\mathrm{d}
\begin{pmatrix}
  1 & x & 1 \\
  1 & 1 & y
\end{pmatrix}
= \begin{pmatrix}
  x & 1 \\
  1 & y
\end{pmatrix}
- \begin{pmatrix}
                             1 & x \\
                             1 & y
\end{pmatrix}
+ \begin{pmatrix}
                             1 & x \\
                             y & 1
\end{pmatrix}.
\end{equation*}

Let $\bb$ be a coordinate of a multi-matrix of an element in degree
$\mathbf{1}_{n-m}$ such
  that $\bb \neq 
  \mathbf{0}_{n-m}$.
  For any multi-matrix $c$ in degree $\mathbf{W} \in
  \mathbb{N}^m$, we can form the following multi-matrix in 
  degree 
  $(\mathbf{W}, \mathbf{1}_{n-m}) \in \mathbb{N}^n$:
  \begin{equation*}
    c_{(-,\mathbf{0})} \otimes y_{(\mathbf{0},\bb)} \text{ has terms }
\begin{cases}
  c_{\ba} & \text{ at coordinate }(\ba, \mathbf{0}_{n-m}); \\
  y_{\bb} & \text{ at coordinate }(\mathbf{0}_m, \bb);\\
  1 & \text{ elsewhere.}
\end{cases}
\end{equation*}

\begin{lem}
The following is a chain map:
\begin{equation*}
  \begin{array}{ccc}
    \mathrm{Tot}(\L^\Q_{T^m}(\Q[t],\Q))  & \to
    & \mathrm{Tot}(\L^\Q_{T^n}(\Q[t],\Q)); \\ 
    c & \mapsto &  c_{(-, \mathbf{0})} \otimes y_{(\mathbf{0},\bb)}.
  \end{array}
\end{equation*}
\end{lem}
\begin{proof}
 Clearly
$
 \mathrm{d}_i(c_{(-,\mathbf{0})} \otimes y_{(\mathbf{0},\bb)}) =
 \mathrm{d}_ic_{(-,\mathbf{0})} 
\otimes y_{(\mathbf{0},\bb)}$   for   $0\leq i \leq m$.
But since the multi-degree of 
$c_{(-,\mathbf{0})} \otimes y_ {(\mathbf{0},\bb)} $
 is $\mathbf{V} =(\mathbf{W}, \mathbf{1}_{n-m}) \in \mathbb{N}^n$ 
 and whenever $v_i=1$, 
 $\mathrm{d}_i= d _{i,0}-  d_{i,1}=0$, we also get
\[
 \mathrm{d}_i(c_{(-,\mathbf{0})} \otimes y_{(\mathbf{0},\bb)}) = 0, 
 \text{ for } m < i \leq n. 
\] 
\end{proof}

This lemma also applies when $y_{(\mathbf{0}, \bb)}$ is replaced by
another multi-matrix that has 
more than one nontrivial term, as long as the nontrivial terms are
all in coordinates of the form 
$(\mathbf{0}_m,\bb)$ for $\bb$ in degree $\mathbf{1}_{n-m}$ and $\bb
\neq \mathbf{0}_{n-m}$. 
It has the following immediate corollary:
\begin{lem}[Orthogonal Moving Lemma]
  \label{lem:moving-2}
  Let $\bb$ be a coordinate in degree $\mathbf{1}_{n-m}$ such
  that $\bb \neq \mathbf{0}_{n-m}$. 
  Let $c, c'$ be elements in multi-degree $\mathbf{W} \in
  \mathbb{N}^m$. If $c \sim c'$ in multi-degree $\mathbf{W}$, then 
\[c_{(-, \mathbf{0}_{n-m})} \otimes y_{(\mathbf{0}_m,\bb)} \sim
  c'_{(-, \mathbf{0}_{n-m})} \otimes y_{(\mathbf{0}_m, \bb)} \]
   in multi-degree $(\mathbf{W}, \mathbf{1}_{n-m})$
\end{lem}

Conceptually, the moving lemmas tell us how to move the nontrivial
elements $x,y$ in certain 
multi-matrices  to lower coordinates. They are stated for a special case
for simplicity, but of course they work for 
any permulation of copies of $\mathbb{N}^n$ in the statement.
The split moving lemma says that  if we have $x_{\mathbf{v}}$ and $y_
{\mathbf{w}} $ where the coordinates share a $1$ in a particular
place, the $1$'s  can be moved to coordinate $0$ separately.  
The orthogonal moving lemma says that the $x$ in $x_{\mathbf{v}}$ and
the $y$ in $y_ {\mathbf{w}} $ can be moved 
separately if they are supported in orthogonal tori (that is, have
their nontrivial entries in different coordinates). 

\begin{prop} \label{prop:relations}
  Let $\one$ and $\two$ be two coordinates of degree $\mathbf{1}_n$.
\begin{enumerate}
\item[(1)] \label{item:1} If $\one$ and $\two$ are both $0$ in the $i$th
  place for some $1 \leq i \leq n$, then 
\[
 x_{\one} \cdot y_{\two} \sim 0. 
\]
  In particular, if $\one \neq \mathbf{1}_n$, then $x_{\one} \sim 0$.
\item[(2)] \label{item:2} In general,
\[x_{\one} \cdot y_{\two} \sim
  \sum_{\substack{\one' \leq \one, \two' \leq \two, \\ \one'+\two' =
      \mathbf{1}_n}} x_{\one'} \cdot y_{\two'}, \]
  where the sum is taken over all coordinates $\one'$ and $\two'$ such that
\begin{itemize}
\item They are place-wise no greater than
  $\one$ and $\two$ respectively;
\item They take $1$ in complementary places.
\end{itemize}
\item[(3)] \label{item:3} For $k \geq 1$ and $n \geq 1$, we have the
  following homologous relation: 
\[ (t^{k})_{\mathbf{1}_n} \sim
  \sum_{\substack{\two_1, \ldots, \two_k \neq \mathbf{0}_n, \\
      \two_1+\ldots+\two_k = \mathbf{1}_n}} 
  \prod_{i=1}^k t_{\two_i}\] 
  In particular, if $k=n$ and we let $\mathbf{e}_i$ denote the
  coordinate that has $1$ at the $i$th place and $0$ at other places, we get
\begin{equation}
\label{eq:interesting-class}
 ( t^{n})  _{\mathbf{1}_n} \sim n! \prod_{i=1}^n t_{\mathbf{e}_i}.
\end{equation}
Also, if $k>n$, this gives us 
\[ (t^{k})_{\mathbf{1}_n} \sim 0\]
\end{enumerate}
\end{prop}
\begin{proof}
The class in \textit{(1)} is a cycle because everything is in multi-degree  $\mathbf{1}_n$ is a cycle; it is null-homologous because it is in the image of the degeneracy $s_{i,0}$ in the $i$th place. 

For \textit{(2)} we write $|\one|$ for the sum of the places of the vector $\one$. We
induct on $|\one|+|\two|$. Notice that a coordinate $\one$ of degree
  $\mathbf{1}_n$ is just a sequence of length 
  $n$ of  $0$'s and  $1$'s and $|\one|$ is just the number of $1$'s in it.

  For $|\one|+|\two| \leq n$, there are two cases: One is that $\one$
  and $\two$ are both $0$ in one 
  place. Then the claim holds because the right-hand side is the empty
  sum and the left-hand side is $0$ by part~\textit{(1)}. The other
  case is that $\one + \two = \mathbf{1}_n$.  Then the claim also
  holds   because the right-hand side has only one copy that is
  exactly the left-hand side. 

  Assume that the claim is true for $|\one|+|\two| \leq m$ where $m
  \geq n$ and suppose now $|\one|+|\two| = m+1$. Since $m+1 \geq n+1$,
  $\one$ and $\two$ have to be both $1$ in some place. Without loss of
  generality, we assume that  
\begin{equation*}
\one = (\mathbf{v_0},1), \ \two = (\mathbf{w_0},1) \text{ where }
\mathbf{v_0}, \mathbf{w_0} \leq \mathbf{1}_{n-1}. 
\end{equation*}
By the Split Moving Lemma (Lemma~\ref{lem:moving-1}), 
\begin{equation*}
x_{\one} \cdot y_{\two} \sim x_{(\mathbf{v_0},0)} \cdot y_{\two} +
x_{\one} \cdot y_{(\mathbf{w_0},0)}. 
\end{equation*}
Since $|(\mathbf{v_0},0)| + |\two| = |\one| + |(\mathbf{w_0},0)| = m$,
by inductive hypothesis we have that
\begin{align*}
   x_{\one} \cdot y_{\two} & \sim
  \sum_{\substack{\mathbf{v_0}' \leq \mathbf{v_0}, \two' \leq \two, \\
  (\mathbf{v_0}',0)+\two' = \mathbf{1}_n}} x_{(\mathbf{v_0}',0)} 
  \cdot y_{\two'} +
  \sum_{\substack{\one' \leq \one, \mathbf{w_0}' \leq \mathbf{w_0}, \\
  \one'+(\mathbf{w_0}',0) = \mathbf{1}_n}} x_{\one'} 
  \cdot y_{(\mathbf{w_0}',0)} \\
  & = \sum_{\substack{\one' \leq \one, \two' \leq \two, \\ \one'+\two'
  = \mathbf{1}_n}} x_{\one'} 
  \cdot y_{\two'}.
\end{align*}

For \textit{(3)} we order the pair $(k,n)$ by the lexicographical
ordering. We induct on $(k,n)$. When $k=1$, the 
claim is trivially true.

Suppose the claim is true for all pairs less than $(k,n)$ where $k \geq 2$.
Taking $\one = \two = \mathbf{1}_n$, $x=t$ and $y=t^{k-1}$ in
part~~\textit{(2)}, we get that 
\begin{equation}
  \label{eq:1}
  ( t^k)_{\mathbf{1}_n}  \sim
  \sum_{\two_1+\one' = \mathbf{1}_n} t_{\two_1}
  \cdot ( t^{k-1}) _{\one'} = 
  \sum_{\substack{\two_1 \neq \mathbf{0}_n \\\two_1+\one' = \mathbf{1}_n}} t_{\two_1}
  \cdot (t^{k-1}) _{\one'}.
\end{equation}
The second step above uses that $t_{\mathbf{0}_n} = 0$ because $t$ is
$0$ in the $\Q[t]$-module $\Q$. Let $m =|\one'|$. By the inductive
hypothesis, we have  
\begin{equation}
  \label{eq:2}
 ( t^{k-1}) _{\mathbf{1}_m} \sim \sum_{\substack{\two'_2, \ldots,
      \two'_{k} \neq \mathbf{0}_m, \\ \two'_2+\ldots+\two'_k =
      \mathbf{1}_m}} 
  \prod_{i=2}^{k} t_{\two'_i}
\end{equation}
For each $\two'_i$ which is a coordinate of degree $\mathbf{1}_m$, we
add in $0$ in places where 
$\one'$ is $0$ to make it a coordinate of degree
$\mathbf{1}_n$. Denote it by $\two_i$. Then the 
Orthogonal Moving Lemma (Lemma~\ref{lem:moving-2}), (\ref{eq:1}) and
(\ref{eq:2}) combine to 
\begin{equation*}
 ( t^{k} ) _{\mathbf{1}_n} \sim
  \sum_{\substack{\two_1, \ldots, \two_k \neq \mathbf{0}_n, \\
      \two_1+\ldots+\two_k = \mathbf{1}_n}} 
  \prod_{i=1}^k t_{\two_i}.
\end{equation*}
\end{proof}

For any $n \geq 2$, we call $t_{\mathbf{1}_n}$ the
\emph{diagonal class} and denote it by $\Delta_n$.  We call
$\prod_{i=1}^n t_{\mathbf{e}_i}$ the \emph{volume form} and denote it by
$\vol_n$. If we include $S^1\hookrightarrow T^n$ as the $i$th
coordinate and identify the first Hochschild homology 
group with the K\"ahler differentials, the generator $dt$ of
 $\HH_1^\Q(\Q[t]; \Q)$  maps to the generator
 we call $dt_i$ in the Loday construction of the torus. In this sense $\vol_n$
corresponds to the degree-$n$ class $dt_1 \wedge \ldots \wedge
dt_n$. 

\begin{proof}[Proof of Theorem \ref{thm:notstable}] 
By Equation \eqref{eq:interesting-class} we know that the map $t \mapsto t^n$
induces a map on $\pi_*\L^\Q_{T^n}(\Q[t]; \Q)$, that sends the diagonal
class, $\Delta_n$, to $n! \vol_n$. Hence, by Lemma
\ref{lem:modulestructure} we know that
\[ \dim_\Q\pi_n(\L_{T^n}^\Q(\Q[t]/t^n; \Q)) <
  \dim_\Q\pi_n(\L_{\bigvee_{k=1}^n\bigvee_{\binom{n}{k}}
    S^k}^\Q(\Q[t]/t^n; \Q)). \]
In particular,
\[ \pi_*(\L_{T^n}^\Q(\Q[t]/t^n; \Q)) \ncong
  \pi_*(\L_{\bigvee_{k=1}^n\bigvee_{\binom{n}{k}} S^k}^\Q(\Q[t]/t^n; \Q)). \] 
\end{proof}
\begin{rem}
  For the non-reduced Loday construction $\L^\Q_{T^n}(\Q[t])$,
  parts~~\textit{(1)} and ~\textit{(2)} of Proposition \ref{prop:relations} are 
  still true. Part~~\textit{(3)} will become
  \[ (t^{k})_{\mathbf{1}_n} \sim \sum_{\two_1+\ldots+\two_k = \mathbf{1}_n}
    \prod_{i=1}^k t_{\two_i} \]
  and Equation~(\ref{eq:interesting-class}) is no longer true. 
\end{rem}

\subsection{$\Q[t]/t^m$ on $T^n$ for $2 \leq m < n$.}
We know that for $\Q[t]/t^n$ we get a discrepancy between $\pi_n$ of the Loday
construction on the $n$-torus and that of the bouquet of spheres that
correspond to the cells of the
$n$-torus. We use this to first show that $\Q[t]/t^m$ causes a similar
discrepancy for $2 \leq m < n$. 

\begin{prop} \label{prop:smalltruncation}
  Let $2 \leq m \leq n$. Then
  \[ \pi_m \L^\Q_{T^n}(\Q[t]/t^m; \Q) \ncong
    \pi_m \L^\Q_{\bigvee_{k=1}^n\bigvee_{\binom{n}{k}} S^k}(\Q[t]/t^m; \Q).\]
\end{prop}
\begin{proof}
  We consider the Tor-spectral sequence
  \[ \Tor_{*,*}^{\pi_*\L^\Q_{T^n}(\Q[t]; \Q)}(\pi_*\L^\Q_{T^n}(\Q[t];
    \Q), \Q) \Rightarrow \pi_*\L_{T^n}^\Q(\Q[t]/t^m; \Q)\] 
  where the $\pi_*\L^\Q_{T^n}(\Q[t]; \Q)$-module structure on
  $\pi_*\L^\Q_{T^n}(\Q[t]; \Q)$ is induced by $t \mapsto t^m$. The $m$-chain
  complex $C_*^{(m)} := \L^\Q_{T^m}(\Q[t]; \Q)$ can be considered as an
  $n$-chain complex whose $m+1, \ldots, n$-coordinates are trivial. Then
  \[ C_*^{(m)} = \L^\Q_{T^m}(\Q[t]; \Q) \hookrightarrow C_*^{(n)} :=
    \L^\Q_{T^n}(\Q[t]; \Q) \]
  is a sub-$n$-complex of $C_*^{(n)}$. We know that $\Delta_m \mapsto m!\vol_m$
  in the homology of the total complex of $C_*^{(m)}$ and hence the same is
  true in $C_*^{(n)}$. Therefore the map 
  \[ \pi_m\L^\Q_{T^n}(\Q[t]; \Q) \ra \pi_m\L^\Q_{T^n}(\Q[t]; \Q)\]
  that is induced by $t \mapsto t^m$ is nontrivial and by Lemma
  \ref{lem:modulestructure} the dimension of $\pi_m\L^\Q_{T^n}(\Q[t]/t^m; \Q)$
  is strictly smaller than the dimension of
  \[ \pi_m\L^\Q_{\bigvee_{k=1}^n \bigvee_{\binom{n}{k}} S^k}(\Q[t]/t^m; \Q). \]
\end{proof}

\subsection{Quotients by polynomials without constant term} 
Let $q(t) = a_mt^m + \ldots + a_1t \in \Q[t]$. Then we can still write
$\Q[t]/q(t)$ 
as a pushout
\[ \xymatrix{\Q[t] \ar[r]^{t \mapsto q(t)} \ar[d]_{t \mapsto 0}& \Q[t]
    \ar@{.>}[d]\\ 
\Q \ar@{.>}[r]& \Q[t]/q(t) } \]
hence the above methods carry over.

\begin{prop}
  Let $m_0$ be the smallest natural number with $1 \leq m_0 \leq m$ with
  $a_{m_0} \neq 0$. Then
  \[ \pi_{m_0}\L^\Q_{T^{m_0}}(\Q[t]/q(t); \Q) \ncong
     \pi_{m_0} \L^\Q_{\bigvee_{k=1}^{m_0}\bigvee_{\binom{m_0}{k}}
       S^k}(\Q[t]/q(t); \Q).\] 
 \end{prop}
\begin{proof}
  If $m_0 = 1$, then $\varepsilon t \in \HH_1^\Q(\Q[t];\Q)$ maps to
  $\varepsilon(q(t))\in \HH_1^\Q(\Q[t];\Q)$ under the map
  $t \mapsto q(t)$. In the module of K\"ahler differentials this
  element corresponds to
  \[ a_1dt + 2a_2tdt + \ldots + ma_mt^{m-1}dt\]
  but all these summands are null-homologous except for the first one. So
  $\varepsilon t \mapsto a_1 \varepsilon t \neq 0$ and this, along  with Lemma \ref{lem:modulestructure},
   proves the
  claim.

We denote by $\Delta_{m_0}(q(t))$ the element
$(q(t))_{\mathbf{1}_{m_0}}$.
If $m_0 > 1$, then the diagonal element $\Delta_{m_0}(t)$ maps to
  \[ \Delta_{m_0}(q(t)) = \sum_{i=m_0}^m a_i\Delta_{m_0}(t^i)\]
  and this is homologous to 
  \[ (m_0)!a_{m_0}\vol_{m_0} + \text{ terms of higher }
    t\text{-degree} \]
by \eqref{eq:interesting-class}. 
  Hence $\Delta_{m_0}(t)$ maps to a nontrivial element and again Lemma
  \ref{lem:modulestructure} gives the claim. 
 \end{proof}

\subsection{Truncated polynomial algebras in prime characteristic}
We know that for commutative Hopf algebras $A$ over $k$ the Loday
construction is 
stable, so Loday constructions of truncated polynomial algebras of the
form $\F_p[t]/t^{p^\ell}$ 
have the same homotopy groups when evaluated on an $n$-torus and on
the corresponding bouquet of 
spheres. However, we show that there is a discrepancy for truncated
polynomial algebras 
$\F_p[t]/t^n$ for $2 \leq  n < p$.

\begin{thm}
  Assume that $2 \leq n < p$ and $n \leq m$, then
  \[ \pi_*(\L_{T^m}^{\F_p}(\F_p[t]/t^n; \F_p)) \ncong
  \pi_*(\L_{\bigvee_{k=1}^m\bigvee_{\binom{m}{k}}
    S^k}^{\F_p}(\F_p[t]/t^n; \F_p)). \]
In particular for all $2 \leq n < p$ the pair $(\F_p[t]/t^n; \F_p)$ is
not stable.  
\end{thm}  
\begin{proof}
  We consider the case $m = n$. The cases $n < m$ follow by an
  argument similar to that for Proposition \ref{prop:smalltruncation}.
  
  As $\F_p[t]$ is smooth over $\F_p$, we know that $\F_p \ra \F_p[t]$
  is stable, so that 
  \[ \pi_*(\L_{T^n}^{\F_p}(\F_p[t]; \F_p)) \cong
    \pi_*(\L_{\bigvee_{k=1}^n \bigvee_{\binom{n}{k}} S^k}^{\F_p}(\F_p[t]; \F_p))
  \cong \bigotimes_{k=1}^n \bigotimes_{\binom{n}{k}}
  \HH^{[k],\F_p}_*(\F_p[t];\F_p)\]  
and $\HH_*^{[k], \F_p}(\F_p[t];\F_p)$ is calculated in \cite[\S
8]{blprz} so that we obtain 
\[ \HH_*^{[k], \F_p}(\F_p[t];\F_p) \cong B'_{k+1}\]
where $B'_1 = \F_p[t]$ and $B'_{k+1} = \Tor_{*,*}^{B'_k}(\F_p, \F_p)$
where the grading on 
$ \Tor_{*,*}^{B'_k}(\F_p, \F_p)$ is the total grading. Thus in low
degrees this gives 
$\HH_*^{\F_p}(\F_p[t]; \F_p) \cong \Lambda_{\F_p}(\varepsilon t)$ with
$|\varepsilon t|=1$, 
$\HH_*^{[2],\F_p}(\F_p[t]; \F_p) \cong
\Gamma_{\F_p}(\varrho^0\varepsilon t)$  with 
$|\varrho^0\varepsilon t|=2$. As $\Gamma_{\F_p}(\varrho^0\varepsilon
t) \cong \bigotimes_{i\geq 0} 
\F_p[\varrho^k\varepsilon t]/(\varrho^k\varepsilon t)^p$ we can
iterate the result. 

Note that in $\HH_n^{[n], \F_p}(\F_p[t];\F_p)$ there is always an 
indecomposable generator of the 
form $\varepsilon \varrho^0 \ldots \varrho^0\varepsilon t$  or
$\varrho^0 \varepsilon \varrho^0 
\ldots \varrho^0\varepsilon t$ in degree $n$ and we call this
generator $\Delta_n$.  
We also obtain a volume class 
\[ \vol_n := \varepsilon t_1 \cdot \ldots \cdot \varepsilon t_n \in
  \pi_n\L_{S^1 \vee \ldots \vee S^1}^{\F_p}(\F_p[t]; \F_p) \hookrightarrow
  \pi_n\L_{T^n}^{\F_p}(\F_p[t]; \F_p). \]
  
The results from Proposition \ref{prop:relations} work over the
integers. If $n < p$, then $n!$ is invertible in $\F_p$ and therefore
the class $\Delta_n$ maps  to $n! \vol_n$. An argument analogous to
Lemma \ref{lem:modulestructure} finishes the proof.  
\end{proof}

\end{document}